\documentclass[11pt,letterpaper]{amsart}
\usepackage{amsmath, amssymb, mathrsfs, mathtools}

\usepackage{stmaryrd}

\usepackage{color}

  \usepackage{graphics} 
  \usepackage{graphicx} 
  \usepackage{epsfig} 
 
\newtheorem{theorem}{Theorem}[section]

\newtheorem{lemma}[theorem]{Lemma}

\newtheorem{proposition}[theorem]{Proposition}
\theoremstyle{definition}
\newtheorem{definition}[theorem]{Definition}
\newtheorem{remark}[theorem]{Remark}

\newcommand{\ep}{\varepsilon}
\newcommand{\eps}{\varepsilon}

\newcommand{\Rb}{\mathbb{R}}
\newcommand{\RR}{\mathbb{R}}
\newcommand{\ZZ}{\mathbb{Z}}
\newcommand{\NN}{\mathbb{N}}

\newcommand{\cbf}{\mathbf{c}}
\newcommand{\be}{\begin{equation}}
\newcommand{\ee}{\end{equation}}
\newcommand{\pa}{\partial}
\newcommand{\Ce}{\mathcal C^e}
\newcommand{\sfS}{\mathsf S}
\newcommand{\sfW}{\mathsf W}

\newcommand{\tu}{\tilde{u}}

\newcommand{\tgamma}{\tilde{\gamma}}

\newcommand{\tU}{\widetilde{U}}
\newcommand{\tJ}{\widetilde{J}}

\newcommand{\lambdabf}{\boldsymbol{\lambda}}

\newcommand{\nubf}{\boldsymbol{\nu}}

\newcommand{\alphabf}{\boldsymbol{\alpha}}

\newcommand{\hdot}{\dot{H}^1}

\newcommand{\OOO}{\mathcal{O}}
\newcommand{\HHH}{\mathcal{H}}

\newcommand{\vv}{\vec{v}}

\newcommand{\indic}{1\!\!1}

\DeclareMathOperator{\wlim}{w-lim}

\numberwithin{equation}{section}

\title[Quadratic wave in 6D] 
{Soliton resolution for the radial quadratic wave equation in space dimension 6}

\author[C.~Collot]{Charles Collot$^1$}
\author[T.~Duyckaerts]{Thomas Duyckaerts$^2$}
\author[C.~Kenig]{Carlos Kenig$^3$}
\author[F.~Merle]{Frank Merle$^4$}
\thanks{$^1$AGM (UMR 8088), Universit\'e de Cergy-Pontoise, and Centre National de la Recherche Scientifique}
\thanks{$^2$LAGA (UMR 7539), Universit\'e Sorbonne Paris Nord, and Institut Universitaire de France}
\thanks{$^3$University of Chicago. Partially supported by NSF Grant DMS-1800082}
\thanks{$^4$AGM (UMR 8088), Universit\'e de Cergy-Pontoise, and Institut des Hautes \'Etudes Scientifiques}

\thanks{ 
\today}
\begin{document}

\begin{abstract}We consider the quadratic semilinear wave equation in six dimensions. This energy critical problem admits a ground state solution, which is the unique (up to scaling) positive stationary solution. We prove that any spherically symmetric solution, that remains bounded in the energy norm, evolves asymptotically to a sum of decoupled mo\-du\-lated ground states, plus a radiation term. As a by-product of the approach we prove the non-existence of multisoliton solutions that do not emit any radiation. The proof follows the method initiated for large odd dimensions by the last three authors, reducing the problem to ru\-ling out the existence of such non-radiative multisolitons, by deriving a contradiction from a finite dimensional system of ordinary differential equations governing their modulation parameters. In comparison, the difficulty in six dimensions is the failure of certain channel of energy estimates and the related existence of a linear resonance. We use the obtention of new channel of energy estimates, from our previous article \cite{CoDuKeMe22}, as well as the classification of non-radiative solutions with small energy, from our work \cite{CoDuKeMe22Pb}.
\end{abstract}

\maketitle

\tableofcontents

\section{Introduction}

In this paper we will consider  the wave equation on $\RR^6$, with the energy-critical focusing nonlinearity:

\begin{equation}
 \label{NLWabs}
 \partial_t^2u-\Delta u=|u|u,
\end{equation} 
together with a similar problem
\begin{equation}
 \label{NLW}
 \partial_t^2u-\Delta u=u^2,
\end{equation} 
where $t\in \RR$ and $x\in \RR^6$, 
with initial data
\begin{equation}
 \label{ID}
 \vec{u}_{\restriction t=0}=(u_0,u_1)\in \HHH,
\end{equation} 
where $\vec{u}=(u,\partial_t u)$, and $\HHH=\hdot(\RR^6\times L^2(\RR^6)$ is the energy space. We will only consider radial initial data, i.e. data depending only on $r=|x|=\sqrt{x_1^2+\ldots+x_6^2}$.

We denote by
$$W(x)=\frac{1}{(1+\frac{|x|^2}{24})^2}$$
the ground state of \eqref{NLWabs} and  \eqref{NLW} which solves $-\Delta W=W^2$.

The equation \eqref{NLWabs} is a special case of the energy-critical wave equation 
\begin{equation}
\label{NLWN}
\partial_t^2u-\Delta u=|u|^{\frac{4}{N-2}}u 
\end{equation} 
in general space dimension $N\geq 3$, whose ground state is given by $W(x)=\left(1+\frac{|x|^2}{N(N-2)}\right)^{1-\frac{N}{2}}$.

\subsection{Background on the soliton resolution conjecture}

\emph{The main results of this paper are the proofs of soliton resolution, without size constraints, and for all times,  for radial solutions of \eqref{NLWabs} and \eqref{NLW} that are bounded in $\HHH$.} The same proof applies to the radial energy critical Yang-Mills equations and the Wave-Maps equations in the $2$-equivariant case. The general non radial problem seems out of reach.

We start with a general discussion of the soliton resolution conjecture for nonlinear dispersive equations. This conjecture predicts that any global in time solution of this type of equation evolves asymptotically as a sum of decoupled solitons (traveling wave solutions, which are well-localized and traveling at a fixed speed), a radiative term (typically a solution to a linear equation) and a term going to zero in the energy space. For finite time blow-up solutions, a similar decomposition should hold, depending on the nature of the blow-up. In the present case, where we consider radial solutions of equations \eqref{NLWabs} and \eqref{NLW} whose energy norm stays bounded, the solitons are the stationary states. The conjecture then predicts that such solutions resolve into a sum of stationary states decoupled by scaling plus a radiation.

This conjecture arose in the 1970's from numerical simulations and the theory of integrable equations (see  \cite{DuKeMaMe21P} for a historic perspective). The first theoretical results in the direction of soliton resolution were obtained for the completely integrable KdV, mKdV and $1$-dimensional cubic NLS, using the method of inverse scattering (\cite{Lax68}, \cite{EcSc83},\cite{Eckhaus86},  \cite{Schuur86BO},  \cite{SegurAblowitz76}, \cite{Novoksenov80}, \cite{BoJeMcL18}). 

For 30/40 years, the conjecture was established with constraints on the initial data, close to a soliton, a setting in which the problem is then perturbative. We refer to the introduction of \cite{DuJiKeMe17a} for a more complete discussion and more references on the subject. The conjecture was also studied in the context of parabolic equations. Classification results for solutions "below the ground state'', i.e. with optimal size constraints on the initial data were obtained in \cite{KeMe08}, \cite{DuMe08} in the case of the energy-critical nonlinear wave equation \eqref{NLWN}  (see \cite{DuKeMaMe21P} for more details). 

 As seen in many recent works, \emph{the proof of rigidity (also called Liouville) theorems, classifying 
solutions that are non-dispersive (in a sense specified below) is crucial in the understanding of the asymptotic dynamics of the semilinear dispersive equation \eqref{NLWN}.} A typical statement is that the only non-dispersive solutions are the stationary solutions (or more generally the solitons) of the equation.
 
A first notion of non-dispersive solutions is given by \emph{solutions with the compactness property in time}, that are solutions whose trajectory is precompact up the invariances of the equation. In the radial case, equations \eqref{NLW} and \eqref{ID} are invariant by the scaling transformation
\be \label{id:def:energyscaling}
u_{(\lambda)}(t,x)=\frac{1}{\lambda^2} u\left(\frac{t}{\lambda},\frac{x}{\lambda} \right)
\ee
in the sense that if $u$ is a solution then so is $u_{(\lambda)}$ for any $\lambda>0$. The concept of solution with the compactness property goes back to \cite{MaMe00}, in the context of the KdV equation (see also \cite{KeMe06} and references therein for NLS). For equation \eqref{NLWN}, these solutions were first considered in \cite{KeMe08}, where a rigidity theorem with a size constraint is proved. The  general rigidity theorem, without a size constraint is proved in \cite{DuKeMe16a} (see also \cite{DuKeMe11a} for the radial, $3D$ case).

 \subsection{Background on the case of hyperbolic equations}

In the context of non-integrable dispersive equations, it became clear that the problem in the hyperbolic situation, especially in the context of energy critical nonlinearities (equation \eqref{NLWN}), was the first to be considered using some decoupling related to the finite speed of propagation. First, results for data close to the ground state were obtained (see \cite{DuKeMe11a}, \cite{DuKeMe12}, \cite{KrNaSc13b}, \cite{KrNaSc15}). Then, the soliton resolution for sequences of times in the radial case, for solutions which are bounded in the energy norm, was proved by \cite{DuKeMe12b} in $3$ dimensions, \cite{Rodriguez16} in all other odd dimensions, in \cite{CoKeLaSc18} in $4$ dimensions and in \cite{JiaKenig17} in $6$ dimensions. In \cite{DuJiKeMe17a}, the second, third and fourth authors, with Hao Jia, proved the decomposition for sequences of times, in the nonradial case, for solutions which are bounded in the energy norm, in dimensions $3$, $4$ and $5$.

To consider the full problem (proving the decomposition for all times) one has to understand the collision of solitons and prove that all collisions produce some radiation, which limits their number by energy considerations.  This is the approach introduced by the last three authors in \cite{DuKeMe13} and fully developed by them in \cite{DuKeMe19Pb,DuKeMe21a,DuKeMe20}. More precisely, the natural object to consider is a pure multisoliton in both time directions, which is a solution that is, asymptotically as $t\to+\infty$ and as $t\to-\infty$, a sum of decoupled solitons without radiation (i.e. the radiation term is zero). For non-integrable equations such as \eqref{NLW}, \eqref{NLWN} and \eqref{WMk} below, it is expected that collisions are inelastic and should always generate some radiation (see e.g. \cite{MartelMerle11b,MartelMerle11,MartelMerle15} in the context of generalized Korteweg-de Vries equations and also \cite{MartelMerle18} for \eqref{NLWN} with $N=5$), ruling out the existence of such an object.

To deal with this problem and using fully the finite speed of propagation, the second, third and fourth authors have introduced the concept of \emph{non-radiative} solutions of \eqref{NLWN}. By definition, these are solutions of \eqref{NLWN}, defined for $|x|>R+|t|$, and such that 
\begin{equation}
\label{non-radiative}
\sum_{\pm} \lim_{t\to\pm\infty}  \int_{|x|>R+|t|} \left(|\nabla u(t,x)|^2+(\partial_tu(t,x))^2\right)dx=0.
\end{equation}

The usefulness of this concept is that, using finite speed of propagation, it can be applied by first studying solutions in the exterior of a wave cone $\{|x|>R+|t|\}$, for large $R$, thus restricting to small solutions, that are close to solutions of the linear wave equation. This is connected with the study of lower bounds of the form
\begin{equation}
 \label{lower_bound}
 C\sum_{\pm}\lim_{t\to\pm\infty} \int_{|x|>R+|t|} |\nabla_{t,x} u_L(t,x)|^2dx\geq \int_{|x|>R} (u_1(x))^2+|\nabla u_0(x)|^2dx
\end{equation} 
for radial solutions of the linear wave equation
\begin{equation}
 \label{LWintro}
 \partial_t^2u_L-\Delta u_L=0, \quad (t,x)\in \RR\times \RR^N,
\end{equation} 
with initial data $\vec{u}_{\restriction t=0}=(u_0,u_1)$. Due to finite speed of propagation, the energy space is $\mathcal H_R$ with norm
$$
\| (u_0,u_1)\|_{\mathcal H_R}^2= \int_{|x|\geq R} (|\nabla u_0(x)|^2+u_1^2(x))dx.
$$

The validity of the linear estimate  \eqref{lower_bound} depends strongly on the dimension $N$ (its size and the oddness/evenness). 

\medskip

- \emph{Odd space dimensions:}

In this case, \eqref{lower_bound} for $R=0$ holds for any $(u_0,u_1)\in \HHH$ (see \cite{DuKeMe12} in the non-radial case).

For $R>0$, the dimension $N=3$ was first considered due to the following exceptional property: \eqref{lower_bound} is valid for all radial initial data $(u_0,u_1) \in \HHH_R)$, that are orthogonal to $\left( r^{-1},0 \right)$. This single degenerate direction can be handled with the scaling invariance \eqref{id:def:energyscaling} of the equation, and corresponds to the asymptotics for large $r$ of the stationary solution $W=\left( 1+r^2/3 \right)^{-1/2}$. This leads to the proof of a strong rigidity theorem: for any $R>0$, the solitons $\pm W_{(\lambda)}$ are the only nonzero solutions to \eqref{NLWN} without radiation at infinity in time in the region $\{|x|>R+|t|\}$ (such property is false for $N\geq 5$, see \cite{CoDuKeMe22Pb}). This leads to the soliton resolution for all radial solutions of \eqref{NLWN} with $N=3$ \cite{DuKeMe13}.

For $N$ odd, $N\geq 5$, \eqref{lower_bound} holds in the radial case, for all radial data in an $\frac{N-1}{2}$ co-dimensional subspace of $\HHH_R$, which is not sufficient to deduce a strong rigidity result  for \eqref{NLWN} as in space dimension $3$,   using the scaling invariance of the equation, see \cite{CoDuKeMe22Pb}. The proof of the soliton resolution in this case is more involved: it combines asymptotic estimates on non-radiative solutions  of \eqref{NLW} deduced from \eqref{lower_bound} with a careful study of the modulation equations close to a multisoliton for non-radiative solutions, which gives enough parameters to deal with the large dimension of the counter examples at infinity. Using a gain of decay in space related to the non-radiative property, the last three authors were reduced to study a finite dimensional dynamics of the scaling parameters of the solitons and were able to prove the soliton resolution for all radial solutions of \eqref{NLW} that are bounded in the energy space, for all times \cite{DuKeMe19Pb,DuKeMe21a,DuKeMe20}.

\medskip

- \emph{Even space dimensions}:

The estimate \eqref{lower_bound} is not valid in its full generality, even when $R=0$. In even space dimensions, up to now, no lower bound of the form \eqref{lower_bound} has been known and counter examples are established in \cite{CoKeSc14} and recently in \cite{CoDuKeMe22}. Nevertheless, \eqref{lower_bound} holds in a finite codimension space (at least in the radial case) for initial data of the form $(u_0,0)$ when $N$ has a congruence to $0$ modulo $4$, or $(0,u_1)$ when $N$ has a congruence to $2$ modulo $4$ (see \cite{CoKeSc14}, \cite{DuKeMaMe21P}, \cite{LiShenWei21P}). In each dimension, for the other case (initial data of the form $(0,u_1)$ or $(u_0,0)$ respectively), one can see this failure as a consequence of the existence of an explicit singular resonant non-radiative solution of \eqref{LWintro}, that fails to be in the energy space by a logarithmic factor. A weaker estimate than \eqref{lower_bound} for this other case, with a logarithmic loss, is given in \cite{CoDuKeMe22Pb}.
 
The four dimensional case was first treated by the last three authors and Martel in \cite{DuKeMaMe21P}. This case turns out to be the critical case for the exceptional property mentioned for $N=3$ above: for $R>0$, the solitons are the only radial non-radiative solutions in the region $\{|x|>R+|t|\}$. 

This property is proved in \cite{DuKeMaMe21P} by a delicate analysis based on the separate study of the projections $u_{\pm}(t)=\frac{1}{2}(u(t)\pm u(-t))$ of the solution $u$ on the vector space of odd (respectively even) in time functions, noticing that the equations satisfied by $u_{\pm}$ are decoupled at first order. The soliton resolution for all radial solutions of \eqref{NLWN} with $N=4$ and also the $k=1$ equivariant wave maps follows (see Remark \ref{re:otherequations} for more details).


\subsection{Main results and ideas of proofs}

In this article, we prove the soliton resolution for all times as well as a Liouville Theorem for non-radiative solutions for the semilinear wave equation on $\RR^6$, with the energy-critical focusing nonlinearity (equations \eqref{NLWabs} and \eqref{NLW}). As a by-product of our methods, as in dimension $N=4$ in \cite{DuKeMaMe21P}, this will give the corresponding soliton resolution and rigidity result for the equivariant energy critical wave map ($k=2$) and the energy critical radial Yang-Mills equations (see Remark \ref{re:otherequations}).

To prove these results, one wants to combine the analysis made in four dimensions with that made in odd dimensions $N\geq 5$, which are of completely different natures. Compared to previous works, we have to overcome the following difficulties:\\
- We are in dimension $N=6>4$ and thus we have to deal with the fact that the set of non-radiative solutions in the exterior of a wave cone at the linear level (that are counter-examples to \eqref{lower_bound}) is of dimension greater than $1$ ($2$ in our case). This in fact leads to the existence of nontrivial radial non-radiative solutions at the nonlinear level (different from a soliton) in regions of the type $\{|x|>R+|t|\}$ (see \cite{CoDuKeMe22Pb} and \cite{CoDuKeMe22Pv1}, Proposition 8.1).\\
- To rule out the possibility of the above counter-examples to emerge from solutions on the whole space, we face a reconnection problem, and one has to prove that the non-radiative extensions of these counter-examples to the region $\{|x|> |t|\}$ are not in the energy space.  More precisely, we have to work as in dimensions $N>4$ and odd, using the analysis close to a multisoliton, and to exclude in this context by contradiction the existence of a regular reconnection. This is highly non-trivial.\\

- As opposed to odd dimensions $N\geq 5$ however, lower bounds of the exterior energy are lacking (\eqref{lower_bound} strongly fail\footnote{In the sense that it does not hold true even in a set of finite codimension.} for data of the form $(u_0,0)$), due to the existence of a resonant direction $(r^{-2},0)$ that barely misses the energy space. To tackle this difficulty, we proved
in \cite{CoDuKeMe22} weaker estimates, where the right-hand side of \eqref{lower_bound} is replaced by a weaker norm of the initial data. The fact that these estimates are weaker makes it more difficult to obtain asymptotic expansions at infinity, as well as to justify the modulation analysis and conclude as in odd dimension. \\
- The low degree of regularity of the nonlinearity $|u|u$ makes the analysis delicate in this case. \\

We obtain the following two theorems for radial solutions to Equations \eqref{NLWabs} and \eqref{NLW}. We introduce the set of radial nonzero stationary solutions:
$$
\mathcal W=\left| \begin{array}{l l l} \{ (\iota W_{(\lambda)},0), \ (\iota,\lambda)\in \{-1,+1\}\times (0,\infty)\} & \mbox{ for Equation } \eqref{NLWabs},\\  \{ ( W_{(\lambda)},0), \ \lambda \in (0,\infty)\} & \mbox{ for Equation } \eqref{NLW}. \end{array} \right.
$$

\begin{theorem}[Rigidity for $6D$ radial critical waves]
 \label{T:rigidity}
Assume that $u$ is a spherically symmetric solution of \eqref{NLWabs} (respectively, of \eqref{NLW}) that is global in time and bounded in energy norm:
$$
\sup_{t\in \mathbb R} \int_{\mathbb R^6} \left((\partial_tu(t,x))^2+|\nabla_x u(t,x)|^2\right)dx<\infty,
$$
and whose initial data $(u_0,u_1)\in \mathcal H$ is not a stationary solution of \eqref{NLWabs} (respectively, of \eqref{NLW}) in the sense that $(u_0,u_1)\notin \mathcal W \cup \{(0,0)\}$. Then there exists $R_0,\eta_0>0$ and $t_0\in \mathbb R$ such that the following holds for all $t>t_0$ or for all $t<t_0$:
 \begin{equation}
  \label{channel}
  \int_{|x|>R_0+|t-t_0|}((\partial_tu(t,x))^2+|\nabla_xu(t,x)|^2) dx\geq \eta_0.
 \end{equation} 

 \end{theorem}
 
Note that Theorem  \ref{T:rigidity} implies the fact that the collision of two or more solitons emits some radiation, and thus that there is no pure multisoliton solution of either Equation \eqref{NLWabs} or Equation \eqref{NLW} in the radial case. 

As a consequence of the rigidity Theorem \ref{T:rigidity} and its proof, we obtain the soliton resolution for these equations.

\begin{theorem}[Soliton resolution for radial $6D$ critical waves]
\label{T:main}
 Let $u$ be a radial solution of \eqref{NLWabs} (respectively, of \eqref{NLW}) and $T_+$ be its maximal time of existence. Assume
\begin{equation}
 \label{IntroCarlos1}
\limsup_{t\uparrow T_+} \int_{\mathbb R^6} \left((\partial_tu(t,x))^2+|\nabla_x u(t,x)|^2\right)dx<\infty.
\end{equation} 
 
Then if $T_+<\infty$, there exist $(v_0,v_1)\in \HHH$, an integer $J\in \NN\setminus\{0\}$, and for each $j\in \{1,\ldots,J\}$, a positive function $\lambda_j(t)$ defined for $t$ close to $T_+$ 
such that
$$ 0<\lambda_J(t)\ll \ldots \ll \lambda_1(t)\ll (T_+-t),\quad \mbox{as } t\to T_+,
$$
and signs $(\iota_j)_{1\leq j \leq J}\in \{-1,+1\}^J$ (respectively, the signs are $(\iota_j)_{1\leq j \leq J}\equiv (1,....,1)$ by convention for Equation \eqref{NLW}), such that
\begin{gather}
\label{expansion_u_bup}
\left\|(u(t),\partial_tu(t))-\left(v_{0}+\sum_{j=1}^J\frac{\iota_j}{\lambda_j^2(t)}W\left(\frac{x}{\lambda_{j}(t)}\right),v_{1}\right)\right\|_{\HHH}\underset{t\to T_+}{\longrightarrow} 0.
\end{gather}
If $T_+=+\infty$, there exists a solution $v_{L}$ of the linear wave equation \eqref{LWintro}, an integer $J\in \NN$, and for each $j\in \{1,\ldots,J\}$, a positive function $\lambda_j(t)$ defined for large $t$ such that 
$$ 0<\lambda_J(t)\ll \ldots \ll \lambda_1(t)\ll t, \quad \mbox{as } t\to +\infty$$
and signs $(\iota_j)_{1\leq j \leq J}\in \{-1,+1\}^J$ (respectively, the signs are $(\iota_j)_{1\leq j \leq J}\equiv (1,....,1)$ by convention for Equation \eqref{NLW}), such that
\begin{equation}
\label{expansion}
\Bigg\|(u(t),\partial_tu(t))-\left(v_{L}(t)+\sum_{j=1}^J\frac{\iota_j}{\lambda_j^2(t)}W\left(\frac{x}{\lambda_{j}(t)}\right),\partial_tv_{L}(t)\right)\Bigg\|_{\HHH}%
\underset{t\to+\infty}{\longrightarrow} 0.
\end{equation}
\end{theorem}


\begin{remark} \label{re:otherequations}
Similar problems to the radial energy critical wave equation are the radial energy critical Yang-Mills equation
\begin{equation}
 \label{YM}
 \partial_t^2u -\partial_r^2u-\frac{1}{r}\partial_ru+\frac{2u(1-u^2)}{r^2}=0,
 \end{equation}
 and the $k$-equivariant wave maps from Minkowski space into the two-sphere which corresponds to solutions of the following equation:
\begin{equation}
 \label{WMk}
 \partial_t^2u -\partial_r^2u-\frac{1}{r}\partial_ru+k^2\frac{\sin(2u)}{2r^2}=0.
 \end{equation}
 Our proof extends readily to Equation \eqref{YM} and Equation \eqref{WMk} with $k=2$, establishing the analogues of the rigidity Theorem \ref{T:rigidity} and of the soliton resolution Theorem \ref{T:main}.

 The case $k=1$ for \eqref{WMk} was first treated by the last three authors and Martel in \cite{DuKeMaMe21P} by the same methods than for the radial critical wave equation \eqref{NLWN} in dimensions $N=4$ (both problems are similar at the linear level). 
 
We also expect that the present methods in dimension $N=6$ extend to higher even dimensions $N$. Note that the new channel of energy estimates are proved in dimensions $6$ and $8$ in \cite{CoDuKeMe22}, but are expected to hold in higher even dimensions, and that 
the classification result holds for all even dimensions for analytic non-linearities, as proved in \cite{CoDuKeMe22Pb}.
\end{remark}

\begin{remark}
Solutions that scatter to linear waves (\eqref{expansion} with $J=0$) and the stationary solutions are examples of solutions to \eqref{NLWabs} for which \eqref{expansion} holds with $J=0$ and $J=1$ respectively.
The construction of a global radial two-soliton ($J=2$) of \eqref{NLWN} with $N=6$ is done in \cite{Jendrej19}. We conjecture that a similar construction can be done for $J>2$ in either the blow-up case or the global case, depending on the dimension, for \eqref{NLWN}. We refer to \cite{DelPinoMussoWei21} for such a construction for the energy-critical heat equation.

The set of initial data whose corresponding solution verifies \eqref{IntroCarlos1} with dynamics as in  Theorem \ref{T:main}, with $J\geq 1$, is expected to be of codimension $1$ (in some sense). The precise description of this set, depending on the space dimension, both in the global in time and and the finite time blow-up cases, is a very delicate open question. See for example
\cite{RaphaelRodnianski12}, \cite{GaoKrieger15}, \cite{RodnianskiSterbenz10}, \cite{Pillai2019P}, \cite{Pillai2020P} and \cite{JeJaLa19P} for the case of wave maps, and \cite{KrScTa09}, \cite{HiRa12}, \cite{KrSc14}, \cite{Jendrej17}, etc, for \eqref{NLWN}.
\end{remark} 
\begin{remark}
 The decomposition result of \cite{DuKeMaMe21P} for $1$-equivariant (co-rotational) wave maps, mentioned in Remark \ref{re:otherequations}, was later extended to all $k\geq 1$, by Jendrej and Lawrie \cite{JendrejLawrie21P}.  The general strategy of \cite{JendrejLawrie21P} is similar to the one introduced in  \cite{DuKeMe20,DuKeMe21a,DuKeMe19Pb}, of proving the inelastic collision of solitons as in the current paper. This strategy gives the passage from a sequential decomposition to a continuous in time one. In \cite{JendrejLawrie21P}, the mechanism for proving the inelastic collision of solitons is not through a rigidity theorem (say in the style of Theorem \ref{T:rigidity}), but through the use of modulation equations (introduced by these authors in a similar context in their work \cite{JendrejLawrie18} on ``two-bubble dynamics for threshold solutions'') combined with a delicate ``no return analysis'', in the neighborhood of a multisoliton, inspired by earlier works in the neighborhood of a single soliton, due to Duyckaerts-Merle \cite{DuMe08}, Nakanishi-Schlag \cite{NaSc11} and Krieger-Nakanishi-Schlag \cite{KrNaSc13b}, \cite{KrNaSc15}. The article \cite{JendrejLawrie21P} was preceded by the works \cite{JendrejLawrie20Pb}, \cite{JendrejLawrie22}, \cite{JendrejLawrie20Pc} for the case of $2$-solitons. In comparing both approaches to establishing the inelastic collision of solitons, one should point out that a rigidity theorem, in the style of Theorem \ref{T:rigidity}, gives quantitative control of the radiation generated by the inelastic collision, coming from the ``outer energy'' lower bound in (say) \eqref{channel}. On the other hand, both approaches yields the non-existence of ``pure multisolitons'', which are solutions which exist for all times such that the ``radiation terms'' $v_L$ in \eqref{expansion}, going to $t=\pm\infty$ are both zero, a fact that reflects the non-elastic collision. Moreover, the ``no return'' approach of \cite{JendrejLawrie21P} bypasses linear estimates such as \eqref{lower_bound}, whose validity in even space dimensions holds for only ``half'' the data. However, as it turns out,  this objection is now removed by \cite{CoDuKeMe22}, in which new estimates in the style of \eqref{lower_bound} are obtained, valid also for all data in even dimension, and which suffice to yield both the rigidity theorem and the full soliton decomposition.   
 
 After the first version of this paper was posted on arXiv in early January 2022, Jendrej and Lawrie \cite{JendrejLawrie22P} posted a new paper on arXiv in March 2022, in which they extend the full soliton resolution proved in Theorem \ref{T:main} for equation \eqref{NLW}, to radial solutions of \eqref{NLWN}, for all $N\geq 4$. The approach in \cite{JendrejLawrie22P} is to prove the inelastic collision of solitons (as in \cite{DuKeMe19Pb}) by the ``no return'' method, as in \cite{JendrejLawrie21P}. See also our discussion in the introduction of \cite{CoDuKeMe22Pb}.
\end{remark}

This paper is a revised version of \cite{CoDuKeMe22Pv1}. It differs from \cite{CoDuKeMe22Pv1} in that Section 3 of that version has been removed and an extension of these results is now available in \cite{CoDuKeMe22}, and in that Sections 4,5 and 8 of \cite{CoDuKeMe22Pv1} have been removed and extensions of these results are now available in \cite{CoDuKeMe22Pb}. In addition, a new version of the crucial Proposition 4.3 in \cite{CoDuKeMe22Pv1}, now using the results in \cite{CoDuKeMe22Pb}, is provided in Proposition \ref{P:nonradiative} of this version.

\subsection{Novelties}

As mentioned earlier, \eqref{lower_bound} holds  for $N$ odd in the radial case, for all initial data in a finite co-dimensional subspace whose co-dimension is $\frac{N-1}{2}$. When $N=3$, the co-dimension is $1$, and this allows for a stronger rigidity statement than in Theorem \ref{T:rigidity}, where $R_0$ is arbitrary. When $N$ is odd, $N\geq 5$, since the co-dimension is larger than $1$, the proofs of the analog of the rigidity theorem (Theorem \ref{T:rigidity}) and of the full soliton resolution, given in \cite{DuKeMe20,DuKeMe21a,DuKeMe19Pb} are much more complicated, involving the study of the modulation parameters through the use of \eqref{lower_bound}. When $N=4$ (see \cite{DuKeMaMe21P} and \cite{LiShenWei21P}), \eqref{lower_bound} holds for data in the form $(u_0,0)$, when $u_0$ is orthogonal to the Newtonian potential $1/r^2$, which is again a co-dimension $1$ subspace. To overcome the lack of any lower bound in \eqref{lower_bound} for data of the form $(0,u_1)$, a new nonlinear object was found, which is an approximate nonlinear solution with data $(0,u_1)$, with $u_1$ barely not in $L^2$, and whose $t$ derivative is non-radiative. Using the analog of \eqref{lower_bound} just described, and this approximate solution to deal with data of the form $(0,u_1)$, a rigidity theorem in the style of Theorem \ref{T:rigidity} was found in \cite{DuKeMaMe21P}, valid for any $R_0>0$, just as in the $N=3$ case, and from this the full soliton resolution followed. 

When $N=6$, the analog of \eqref{lower_bound}, established in \cite{DuKeMaMe21P} (see also \cite{LiShenWei21P} and Proposition \ref{P:CK1.4} below), holds for data of the form $(0,u_1)$, with $u_1$ orthogonal in $L^2(\{r>R\})$ to $1/r^4$, the Newtonian potential in $\RR^6$. This is, of course, again ``half'' the data, as in $\RR^4$. However, in addition, $(0,1/r^4)$ does not ``correspond'' to the initial data of a static solution, as $(1/r^2,0)$ does in $\RR\times \RR^4$. This is similar to the $N=5$ situation, and hence a combination of the methods of \cite{DuKeMe20,DuKeMe21a,DuKeMe19Pb} and \cite{DuKeMaMe21P} is required. But, in order to obtain the estimates on the modulation parameters needed in \cite{DuKeMe19Pb}, in order to prove the analog of Theorems \ref{T:rigidity} and \ref{T:main}, dispersive estimates of the type of \eqref{lower_bound} were still needed. This issue was resolved in \cite{CoDuKeMe22}. This is through the use of a weaker version of \eqref{lower_bound}, which still gives quantitative dispersive estimates, but with a logarithmic loss, valid in even dimensions, for data in the complement of a finite dimensional space valid for solutions of the linearized equation around the soliton $W$ (see \S \ref{subsec:channelsunsoliton}) and for solutions of the linearized equation around a multisoliton (see \S \ref{subsec:channels_multi}). These (weaker) dispersive estimates, valid for all data (up to finite co-dimension), for all even dimensions, suffice to obtain the needed estimates for the modulation parameters, and establish Theorem \ref{T:rigidity} and Theorem \ref{T:main}. 

Another important novelty of this work, fully developed in the companion paper \cite{CoDuKeMe22Pb}, explains the difference between the stronger rigidity theorems valid when $N=3,4$, and the slightly weaker version given in Theorem \ref{T:rigidity} and its analog in \cite{DuKeMe19Pb}, valid for a chosen $R_0$. 
The work \cite{CoDuKeMe22Pb} classifies in dimensions $N\geq 3$ non-radiative solutions in regions of the form $\{r>R+|t|\}$, showing they belong to an $\lfloor \frac{N-1}{2}\rfloor$ dimensional family of solutions. For $N\geq 5, \lfloor \frac{N-1}{2}\rfloor\geq 2$ and there exist non-radiative solutions that are not stationary solutions, unlike the cases $N=3,4$, and hence the stronger version of Theorem \ref{T:rigidity} fails when $N\geq 5$ (see also \cite{CoDuKeMe22Pv1}, Proposition 8.1, for $N=6$).

We believe that the new ideas explained above and developed in this paper and in \cite{CoDuKeMe22Pb}, \cite{CoDuKeMe22} will have a wide range of applicability.

\section{Acknowledgements}

This work was supported by the National Science Foundation [DMS-2153794 to C.K.]; the CY Initiative of Excellence Grant
"Investissements d'Avenir" [ANR-16-IDEX-0008 to C.C. and F.M.]; and the France and Chicago Collaborating in the Sciences [FACCTS award \#2-91336 to C.C. and F.M.]

\section{Preliminaries}
In this preliminary section, we recall results on local-wellposedness for equations \eqref{NLWabs} and \eqref{NLW} and channels of energy estimates for the linearized equation around a multisoliton (from \cite{CoDuKeMe22}. We also prove, as a consequence of the classification non-radiative solutions  for equations \eqref{NLWabs} and \eqref{NLW} outside wave cones obtained in \cite{CoDuKeMe22Pb}, some asymptotic estimates on non-radiative solutions that are crucial in the proof of Theorems \ref{T:rigidity} and \ref{T:main}. 
\subsection{Notations and local well-posedness}

We start with some notation.
If $u$ is a function of space and time, we write $\vec{u}=(u,\partial_tu)$. 

For $R\geq 0$, $p\in (1,\infty)$ we write 
\begin{gather*}
\|u\|^p_{L^p_R}=\int_{R}^{\infty} (u(r))^pr^5dr,\quad 
\|u\|^2_{\dot{H}^1_R}=\int_{R}^{\infty}(\partial_ru(r))^2r^5dr,\\
\|\psi\|_{L^2(R,\infty)}=\int_{R}^{\infty}(\psi(r))^2dr.
\end{gather*}
\begin{remark}
\label{R:ext}
Let $R>0$ and $u$ be a radial function defined for $r>R$. Then the extension $u_R$ of $u$ defined by
$$ u_R(r)=u(r),\; r>R,\quad u_R(r)=3u(2R-r)-2u(3R-2r), \; 0<r<R,$$
satisfies, for all $p\geq 1$
$$\|u_R\|_{L^p(\RR^6)}\leq C\|u\|_{L^p_R},\quad \|\partial_ru_R\|_{L^p(\RR^6)}\leq C\|\partial_r u\|_{L^p_R}$$
where the constant $C$ is independent of $u$, $p$ and $R$.
\end{remark}

For $(t,R)\in \Rb\times (0,\infty)$, we let 
$$\Ce_{t,R}=\Big\{ \left(\overline{t},\overline{r}\right)\in \RR\times (0,\infty)\,:\, \overline{r}>R+|t-\overline{t}|\Big\}$$
be the exterior cone. The exterior energy is 
$$\|u\|_{E_{t,R}}=\sup_{\overline{t}\in \RR} \left\|\vec{u}\left( \overline{t} \right)\right\|_{\dot{H}^1_{R+|t-\overline{t}|}\times L^2_{R+|t-\overline{t}|}}.$$
We will write $u\in C_t(I,\dot{H}^1_{R_0+|t-t_0|})$ when $u$ is the restriction to $\Ce_{t_0,R_0}$ of a function $u\in C(I,\dot{H}^1)$. We will use a similar notation, with the same meaning, for other time dependent spaces (e.g. $L^2_{R+|t-t_0|}$).

We introduce the Strichartz norms:
\begin{equation}
 \|u\|_{L^p_tL^q_r(r>R+|t|)}=\left( \int_{\overline{t}\in \RR} \left( \int_{r>R+|t-\overline{t}|}\left|u(\overline{t},r)\right|^qr^5dr \right)^{\frac pq}d\overline{t} \right)^{\frac 1p}=\|u\|_{L^pL^q\left( \Ce_{t,R} \right)}.
\end{equation} 
We will also need Strichartz norms over Besov spaces. We will follow the definitions and results in sections 1,2 of \cite{DuKeMe21a}. 
\begin{gather*}
\sfS:= L^{\frac 72}(\RR^{7}),\quad 
\sfW:=L^{\frac{14}{5}}\left( \RR,\dot{B}^{\frac{1}{2}}_{\frac{14}{5},2} (\RR^6) \right)\\ 
\sfW':= L^{\frac{14}{9}}\left(\RR,\dot{B}^{\frac{1}{2}}_{\frac{14}{9},2}(\RR^6) \right).
\end{gather*}
We recall the Strichartz estimates: if $(u_0,u_1)\in \dot{H}^1\times L^2(\RR^6)$ and 
$$u(t)=\cos t\sqrt{-\Delta}u_0+\frac{\sin t\sqrt{-\Delta}}{\sqrt{-\Delta}}u_1+\int_0^t \frac{\sin(t-t')}{\sqrt{-\Delta}} f(t')dt',$$
with $f=f_1+f_2$, where $f_1\in W'$ and $f_2\in L^1_tL^2_r(\RR^6)$, we have
\begin{equation}
\label{CK1}
 \sup_{t\in \RR}\left\|\vec{u}(t)\right\|_{\dot{H}^1\times L^2}+\|u\|_{\sfS}+\|u\|_{\sfW}+\|u\|_{L^2_tL^{4}_r}
 \lesssim \|(u_0,u_1)\|_{\dot{H}^1\times L^2}+\|f_1\|_{\sfW'}+\|f_2\|_{L^1_tL^2_r}.
\end{equation}
By definition, a solution of \eqref{NLWabs} (respectively \eqref{NLW}) with $t_0\in I$ and $\vec{u}_{\restriction t=t_0}=(u_0,u_1)$ is a $u\in C(I,\dot{H}^1)$, with $\partial_t u\in C(I,L^2)$ such that 
$$\forall t\in I,\quad u(t)=S_L(t-t_0)(u_0,u_1)+\int_{t_0}^{t} \frac{\sin\left( (t-s)\sqrt{-\Delta} \right)}{\sqrt{-\Delta}} F(u(s))ds,$$
with $F(u)=|u|u$ (respectively $F(u)=u^2$).
\begin{proposition}[local well-posedness]
\label{P:LWP}
There exists a constant $\theta$, $0<\theta<1$ such that the following holds.
 Assume $\|(u_0,u_1)\|_{\dot{H}^1\times L^2}\leq A$. If $\|S_L(t)(u_0,u_1)\|_{\sfS}\leq \eta$, for $\eta\leq \eta(A)$ small enough, then there exists a unique solution $u$ of \eqref{NLWabs} (respectively \eqref{NLW}) in $C(\RR,\dot{H}^1\times L^2)$, and
 $$ \forall t\in \RR,\quad \left\|\vec{u}(t)-\vec{S}_L(t)(u_0,u_1)\right\|_{\dot{H}^1\times L^2} \leq C\eta^{\theta}A^{1-\theta}.$$
\end{proposition}
For this result, see for example \cite{BuCzLiPaZh13}. In the radial case the unconditional uniqueness holds in all dimensions, see \cite{DuKeMe21a}. As a consequence of Proposition  \ref{P:LWP}, we have
\begin{multline}
 \label{CK+}
 \|(u_0,u_1)\|_{\dot{H}^1\times L^2}\leq \ep_0\\
 \Longrightarrow 
 \sup_t \|\vec{u}(t)\|_{\dot{H}^1\times L^2}+ \|u\|_{\sfS}+\|u\|_{\sfW}+\|u\|_{L^2_tL^{4}}\lesssim\|(u_0,u_1)\|_{\dot{H}^1\times L^2}.
\end{multline} 
We will next recall the theory for exterior cones, from \cite[Section 2]{DuKeMe21a}. If $\Omega$ is an open set in $\RR^d$, ($d=6$ or $d=7$), and $A$ is a Banach space of distributions in $\RR^d$, we set $\|u\|_{A(\Omega)}=\inf_{\overline{u}} \|\overline{u}\|_{A}$, here the infimum is taken over all $\overline{u}$ such that $\overline{u}_{\restriction \Omega}=u$ (where ${\cdot}_{\restriction \Omega}$ means the restriction in the sense of distributions).

We recall from \cite[Lemma 2.3]{DuKeMe21a} that the characteristic function $\indic_{\{|x|>R\}}$ is a multiplier of $\dot{B}^{1/2}_{\frac{14}{9},2}(\RR^6)$ (to itself) and that $\indic_{\Ce_{0,R_0}}$ is a pointwise multiplier from $W'$ to itself. As a corollary we have, for any $R_0>0$ with $F(u)=u^2$, 
$$\left\|\indic_{\Ce_{0,R_0}}F(u)\right\|_{\sfW'}\lesssim \|u\|_{\sfS\left( \Ce_{0,R_0} \right)}\|u\|_{\sfW\left( \Ce_{0,R_0} \right)}$$
(\cite[Remark 2.4]{DuKeMe21a}).
\begin{definition}
 Let $(u_0,u_1)\in \dot{H}_{R_0}\times L^2_{R_0}$. A solution $u$ of \eqref{NLW} on $\Ce_{0,R_0}$ with initial data $(u_0,u_1)$ is the restriction to $\Ce_{0,R_0}$ of a solution $\tilde{u}\in C\left( \RR,\dot{H}^1 \right)$, with $\partial_t\tilde{u}\in C\left( \RR,L^2 \right)$ of the equation $\partial_t^2\tilde{u}-\Delta\tilde{u}=\tilde{u}^2\indic_{\Ce_{0,R_0}}$ with initial data $(\tilde{u}_0,\tilde{u}_1)\in \dot{H}^1\times L^2$, where $(\tilde{u}_0,\tilde{u}_1)$ is an extension of $(u_0,u_1)$. 
\end{definition}
Note that by finite speed of propagation, $u_{\restriction \Ce_{0,R_0}}$ does not depend on the extension $(\tu_0,\tu_1)$. 
\begin{lemma}
\label{L:CK1.5}
 For $A>0$, there exists $\eta=\eta(A)$ such that if $$(u_0,u_1)\in \dot{H}^1_{R_0}\times L^2_{R_0}, \; \|(u_0,u_1)\|_{\dot{H}^1_{R_0}\times L^2_{R_0}}\leq A\text{ and }\|S_L(t)(u_0,u_1)\|_{\sfS\left( \Ce_{0,R_0} \right)}\leq \eta,$$ then there exists a unique solution $u$ to \eqref{NLWabs} (respectively \eqref{NLW}) in $\Ce_{0,R_0}$ and 
 $$ \forall t,\quad \left\|\vec{u}(t)-S_L(t)(u_0,u_1)\right\|_{\dot{H}^1_{R_0+|t|}\times L^2_{R_0+|t|}}\leq C\eta^{\theta}A^{1-\theta}.$$ 
 (Above $\vec{S}_L(t)(u_0,u_1)_{\restriction \Ce_{0,R_0}}=\vec{S}_L(t)(\tu_0,\tu_1)$ for any extension $(\tu_0,\tu_1)$ of $(u_0,u_1)$). 
 Moreover, the corresponding estimates to \eqref{CK+} hold, that is:
\begin{equation}
 \label{CK+local}
 \|\vec{u}(t)\|_{E_{t_0,R_0}}+ \|u\|_{\sfS(\Ce_{t_0,R_0})}+\|u\|_{\sfW(\Ce_{t_0,R_0})}+\|u\|_{L^2_tL^{4}(\Ce_{t_0,R_0})}\lesssim\|(u_0,u_1)\|_{\dot{H}^1_{R_0}\times L^2_{R_0}}.
\end{equation}  
\end{lemma}
See \cite[Proposition 2.6]{DuKeMe21a}, and also Remark 2.7 there for the unconditional uniqueness in the radial case.

We conclude this subsection with the definition of non-radiative solutions:
\begin{definition}
 \label{D:CK1.1}
 We say that a function $u$ defined on $\Ce_{t_0,R_0}$, such that $E_{t_0,R_0}$ is finite is $(t_0,R_0)$ \emph{non-radiative} if 
 $$\lim_{t\to\pm \infty} \left\|\vec{u}(t)\right\|_{\dot{H}^1_{R_0+|t-t_0|}\times L^2_{R_0+|t-t_0|}=0}.$$
 If in addition $\left\|\vec{u}(t_0)\right\|_{\dot{H}^1_{R_0}\times L^2_{R_0}}\leq \ep_0,$ $\ep_0$ as in \eqref{CK+} below for exterior cones, we will say that $u$ is a small non-radiative solution. $(t_0,R_0)$ will be explicit from the context.
\end{definition}

\subsection{Channels of energy close to a multisoliton in space dimension $6$}
\label{S:channels_lin}
In this subsection, we recall lower bounds of the exterior energy, for solutions of the linearized equation around a soliton or a multisolitons, proved in \cite{CoDuKeMe22}. We start with a lower bound for odd solutions of the free wave equation that will be needed in the sequel:
 \begin{proposition}[Channels in 6d, with right-hand side]
 \label{P:CK1.4}
 Let $u\in \RR\times \RR^6$, solve 
\begin{equation*}
 \Box u=f,\quad u_{\restriction t=0}=0,\quad \partial_tu_{\restriction t=0}=u_1,
\end{equation*} 
$f,u_1$ radial. Fix $R>0$ and write $u_1=\frac{c_1}{r^4}+u_1^{\bot}$, where $\int_{R}^{\infty} u_1^{\bot}\frac{1}{r^4}r^5dr=0$. Then,
\begin{equation}
 \label{CK1.1}
 \left\|u_1^{\bot}\right\|_{L^2_R}\lesssim \lim_{t\to\infty} \left\|\vec{u}(t)\right\|_{\dot{H}^1_{R+|t|}\times L^2_{R+|t|}}+\|f\|_{L^1_tL^2_r\left( \Ce_{0,R} \right)}.
\end{equation}
 \end{proposition}
\begin{proof}
 Let $\tu$ be the solution of the homogeneous wave equation with initial data $(u_0,u_1)$, and let $v$ be the solution of the inhomogeneous equation with right-hand side $f\indic_{\Ce_{0,R}}$, and $(0,0)$ initial data. Then $u=\tu+v$ on $\Ce_{0,R}$, by finite speed of propagation. By \cite[Proposition 3.8]{DuKeMaMe21P}, we have  
 \begin{multline*}
\left\|u_1^{\bot}\right\|_{L^2_R}=\left\|\tu_1^{\bot}\right\|_{L^2_R}\leq \frac{20}{3}\lim_{t\to\infty} \left\| \vec{u}-\vec{v}(t)\right\|_{\dot{H}^1_{R+|t|}\times L^2_{R+|t|}}\\
\leq \frac{20}{3} \left( \lim_{t\to\infty} \left\|\vec{u}(t)\right\|_{\dot{H}^1_{R+|t|}\times L^2_{R+|t|}}+\sup_{t}\left\|\vec{v}(t)\right\|_{\dot{H}^1_{R+|t|}\times L^2_{R+|t|}}\right)\\
\leq \frac{20}{3} \lim_{t\to\infty} \left\|\vec{u}(t)\right\|_{\dot{H}^1_{R+|t|}\times L^2_{R+|t|}}+C\|f\|_{L^1_tL^2_r\left( \Ce_{0,R} \right)}.
 \end{multline*}
\end{proof}

\subsubsection{Channels of energy around the ground state} \label{subsec:channelsunsoliton}
We consider the linearised equation, with $V=-2W$:
\begin{gather}
 \label{LW}
 \partial_t^2u_L-\Delta u_L+V u_L=0\\
 \label{IDLW}
 \vec{u}_{L\restriction t=0}=(u_0,u_1),
\end{gather} 
where $(u_0,u_1)\in \mathcal H=\dot H^1 \times L^2(\mathbb R^6)$. Note that this is the linearised equation for Equations \eqref{NLW} and \eqref{NLWabs} around $W$, as well as for Equation \eqref{NLWabs} around $-W$. It is easy to check that $u$ is globally well-posed in $\HHH$. Indeed, the local well-posedness can be proved by Strichartz estimates and the fact that $W$ is in $L^4$. The global well-posedness follows from the linearity of the equation.

We introduce the orthogonality direction:
$$
\Phi=2W\Lambda W =-\Delta \Lambda W
$$
We let
$$
\mathcal H_R=\dot H^1_R\times L^2_R
$$
We define for $\alpha \in \mathbb R$:
\begin{equation}
\label{defZ}
\| f\|_{Z_{\alpha}}=\sup_{R>0} \frac{R^{-3-\alpha}}{\langle \log R \rangle}\| f\|_{L^2(R\leq r\leq 2R)}, 
\end{equation} 
and note that this norm captures a $\langle \log r \rangle r^\alpha$-type behaviour, with in particular $\| \langle \log r \rangle r^{\alpha}\|_{Z_\alpha}$ finite, and $\| \langle \log r \rangle^{s} r^{\alpha}\|_{Z_\alpha}=\infty$ for $s>1$. Let $Z_\alpha$ be the space of radial $L^2_{\mathrm{loc}}$ functions in $\RR^6$ such that this norm is finite. Note that from Sobolev embedding, $\dot H^1 \subset Z_{-2}$ with:
$$
\| u \|_{Z_{-2}}\lesssim \| u \|_{\dot H^1}.
$$
Note also that $\| u \|_{Z_{-3}}\lesssim \| u\|_{L^2}$. Using Cauchy-Schwarz and the formula $f(r)=-\int_r^{\infty} \partial_rf(s)ds$, one can also prove the following variant of Hardy's inequality:
 $$ \|u\|_{Z_{-2}}\lesssim \|\nabla u\|_{Z_{-3}},$$
 for any $u\in \dot{H}^1$ radial.
 
If $H$ is a Hilbert space, and $E$ a closed linear subspace of $H$, we denote by $\Pi_H(E)$ the orthogonal projection onto $E$ in $H$. We then define the projections:
$$
\Pi_{\dot H^1}^\perp=\Pi_{\dot H^1} (\text{Span}(\Lambda W))^\perp, \qquad \Pi_{L^2}^\perp=\Pi_{L^2} (\text{Span}(\Lambda W))^\perp.
$$
\begin{proposition}[Channels of energy around the ground state] \label{pr:channelsunsoliton}
There exists $C>0$ such that any radial solution $u$ of \eqref{LW}, \eqref{IDLW} with $(u_0,u_1)\in \mathcal H$ satisfies:
\be \label{bd:channelunsoliton}
\big\| \Pi_{L^2}^\perp u_1\big\|_{L^2}^2+\big\|\nabla \Pi_{\dot H^1}^\perp u_0 \big\|_{Z_{-3}}^2\leq C \sum_{\pm } \lim_{t\rightarrow \pm \infty}  \int_{r\geq |t|} |\nabla_{t,x}u_L(t,x)|^2dx.
\ee
\end{proposition}
See \cite[Theorem 1.1]{CoDuKeMe22}.

\subsubsection{Channels of energy close to a multisoliton}
\label{subsec:channels_multi}

We introduce for $J\in \mathbb N$:
$$
\Lambda_J=\left\{\lambdabf=(\lambda_1,...,\lambda_J)\in (0,\infty)^J, \ \lambda_J< \lambda_{J-1}<...<\lambda_{1} \right\}
$$
and define for all $\lambdabf \in \Lambda_J$ the scale separation parameter:
$$
\gamma (\lambdabf)=\max_{1\leq j\leq J-1}\frac{\lambda_{j+1}}{\lambda_j}.
$$
We shall use the convention that for $\lambdabf \in \Lambda_J$, $
\lambda_{J+1}=0$ 
and $\lambda_0=\infty$.
We define for $\lambda>0$ the $\dot H^1$ and $L^2$ rescalings:
$$
f_{(\lambda)}=\frac{1}{\lambda^2}f\left(\frac{r}{\lambda}\right), \qquad f_{[\lambda]}=\frac{1}{\lambda^3}f\left(\frac{r}{\lambda}\right).
$$
Given $\lambdabf \in \Lambda_J$ we define the potential around a multisoliton:
$$
V_{\lambdabf}=  \sum_{j=1}^J V_{(\lambda_j)},
$$
where we recall $V=-2W$. We study in this subsection solutions to:
\begin{equation}
 \label{id:linearmultisoliton} 
 \left\{
 \begin{aligned}
 \partial_t^2u-\Delta u+V_{\lambdabf}u=0,\\
 \vec{u}_{\restriction t=0}=(u_0,u_1)\in \mathcal H,
 \end{aligned}\right.
\end{equation}
and assume throughout that $u$ is radially symmetric. We define, for $\lambdabf \in \Lambda_J$:
\begin{gather*}
\| f\|_{Z_{\alpha,\lambdabf}}=\sup_{R>0} \ \frac{R^{-3-\alpha}}{\inf_{1\leq j \leq J} \ \langle \log \frac{R}{\lambda_j} \rangle}\| f\|_{L^2(R\leq r\leq 2R)}.
\end{gather*}
We note that this norm captures a $ r^\alpha$-type behaviour with logarithmic loss away from the solitons, with in particular $\|r^{\alpha} \ \inf_i \langle \log r \lambda_i^{-1} \rangle \|_{Z_{\alpha,\lambdabf}}\approx 1$. Let $Z_{\alpha,\lambdabf}$ stand for the Banach spaces of radial functions associated with this norm. Note that from Sobolev embedding, $\dot H^1\subset Z_{-2,\lambdabf}$ with the following estimates that are uniform in $\lambdabf$:
$$
\| u \|_{Z_{-2,\lambdabf}}\lesssim \| u \|_{\dot H^1}.
$$
Recall the notation for the projectors $\Pi_{L^2}(E)$ and $\Pi_{\dot H^1}(E)$ of Subsection \ref{subsec:channelsunsoliton}. We define:
\begin{align*}
&\Pi_{\dot H^1,\lambdabf}= \Pi_{\dot H^1}  \left( \text{Span}((\Lambda W)_{(\lambda_j)})_{1\leq j \leq J} \right), \qquad \Pi_{L^2,\lambdabf}= \Pi_{L^2}  \left( \text{Span}((\Lambda W)_{[\lambda_j]})_{1\leq j \leq J} \right),\\
&\Pi_{\dot H^1,\lambdabf}^\perp=\operatorname{Id}-\Pi_{\dot H^1,\lambdabf}, \qquad \qquad \qquad \qquad \qquad \Pi_{L^2,\lambdabf}^\perp=\operatorname{Id}-\Pi_{L^2,\lambdabf}.
\end{align*}

\begin{proposition}[Channels of energy around a multisoliton] \label{pr:channels}

For any $J\in \mathbb N$, there exist $\gamma^*,C>0$ such that for any $\lambdabf \in \Lambda_J$ with $\gamma(\lambdabf)\leq \gamma^*$ if $u$ solves \eqref{id:linearmultisoliton} on $\mathbb R^{1+6}$ then:
\begin{multline} \label{bd:channelsmultisoliton}
\| \Pi_{L^2,\lambdabf}^\perp \, u_1 \|_{L^2}^2+\| \nabla \Pi_{\dot H^1,\lambdabf}^\perp \, u_0 \|_{Z_{-3,\lambdabf}}^2\\
\leq C\left( \sum_{\pm } \lim_{t\rightarrow \pm \infty}  \int_{r\geq |t|} |\nabla_{t,x}u|^2 \ +\gamma (\lambdabf)^2 \| (u_0,u_1)\|_{\mathcal H}^2\right).
\end{multline}
\end{proposition}
See \cite[Theorem 1.4]{CoDuKeMe22}.

\subsection{Non-radiative solutions}
This section is devoted to non-radiative solutions, as defined in Definition \ref{D:CK1.1}. We will first give general properties of these solutions on $\Ce_{t,R}$, $R$ large. These properties are direct consequences of \cite{CoDuKeMe22Pb}. In Subsection \ref{Sub:close_multisoliton}, we will study non-radiative solutions close to a multisoliton, using the exterior energy estimates for the linearized equation of Section \ref{S:channels_lin}.

\subsubsection{Far away properties of non-radiative solutions}
\label{Sub:nonradiative}

In this subsection, we give properties on non-radiative solutions of equations \eqref{NLW} and \eqref{NLWabs} on $\Ce_{t_0,R_0}$ with small energy. Note that for a fixed non-radiative solution, the small energy assumption is always satisfied provided $R_0$ is chosen large enough.
We start by stating the existence of a negative stationary solution of \eqref{NLW} defined for large $r$. We denote by $c_W=24^2$, so that 
$$ \lim_{r\to\infty} r^4 W(r)=c_W.$$
\begin{lemma}
 \label{L:EllExistence}
 There exists $R_->0$ and $W^-\in C^{\infty}(\RR^6\cap\{|x|>R_-\})$, radial, such that
 \begin{gather}
  \label{Ell11}
  -\Delta W^-=\left( W^- \right)^2,\quad r>R_-\\
  \label{Ell12}
  \lim_{r\to R_-} W^-(r)=-\infty.
 \end{gather} 
 and $W^-$ satisfies, for large $r$,
 \begin{equation}
  \label{Ell10}
  \left|W^{-}(r)+\frac{c_W}{r^4}\right|\lesssim \frac{1}{r^6},\quad \left|\frac{d W^-}{dr} -\frac{4c_W}{r^5}\right|\lesssim \frac{1}{r^7}.
 \end{equation} 
\end{lemma}
\begin{remark}
 As a consequence of \eqref{Ell12}, \eqref{Ell10}, we have $W^-\in \dot{H}^1_R$ for all $R>R_-$ and $W^-\notin \dot{H}^1_{R_-}$. 
\end{remark}

\begin{proof}[Sketch of proof of Lemma \ref{L:EllExistence}]
(See \cite{DuKeMe12c}, \cite{DuyckaertsRoy17}).

One can prove the existence of $W^-$ using the following Duhamel form of the equation \eqref{Ell11},
\begin{equation}
\label{FPW-} 
W^-(r)=-\frac{c_W}{r^4}-\int_{r}^{\infty}\frac{1}{\rho^5} \int_{\rho}^{\infty} \left( W^-(s) \right)^2s^5ds
\end{equation} 
by fixed point 
in the metric space 
$$\left\{ f\in C^0([R,\infty)), \quad N_R(f):=\max_{r\geq R} r^4|f(r)|\leq 2c_W\right\},$$
where $R$ is large, 
with the metric induced by the norm $N_R$. 
The fact that $W^-$ is $C^{\infty}$, and the estimate \eqref{Ell10} follow easily from \eqref{FPW-}.

We denote by $(R_-,\infty)\subset (0,\infty)$ the maximal interval of existence of $W^-$, as a solution of the ordinary differential equation $y''+\frac 5r y'+y^2=0$. To prove that $R_->0$, we argue by contradiction, assuming that $R_-=0$. Let $s=1/r^4$, and define $Z$ by $W_-(r)=Z(s)=Z(1/r^4)$, so that $Z$ is defined on $(0,\infty)$ and
\begin{equation}
 \label{eqZ}
 Z''+\frac{1}{16 s^{3/2}}Z^2=0,\quad \lim_{s\to 0} \frac{1}{s}Z(s)=\lim_{s\to 0} Z'(s)=-c_W.
\end{equation} 
By \eqref{eqZ}, $Z'(s)\leq -c_W$, $Z(s)\leq -c_Ws$ for all $s$. By a straightforward induction and \eqref{eqZ}, one also proves that $|Z(s)|\gtrsim s^n$ for all $n$ and $s$. In particular, there exists a constant $C>0$ such that for $s$ large, $|Z''(s)|\geq |Z|^{3/2}$. Together with the facts that $Z,Z'$ and $Z''$ are negative, one can deduce blow-up in finite time by standard arguments, yielding a contradiction. Since $Z(s)$ is negative, \eqref{Ell12} follows by a standard blow-up criterion for differential equations. The proof of the lemma is complete.
\end{proof}

In the case of equation \eqref{NLWabs}, we denote $W^-=-W$, which also satisfies the estimate \eqref{Ell10}.

We write $W^+=W$. The main result of this section is the following for non-radiative solutions as defined in Definition \ref{D:CK1.1}:
\begin{proposition}
 \label{P:nonradiative}
 There exist constants $C>0$, $\eps_0$ with the following properties. Let $u$ be a solution of \eqref{NLWabs} or \eqref{NLW} defined on an interval $I$, which is a $(t_0,0)$ non-radiative solution for all $t_0\in I$. 
Then one of the following holds:
 \begin{enumerate}
  \item \label{caseW} The solution $u$ is stationary. In other words, $u\equiv 0$ or there exists $\lambda>0$ and a sign $\iota\in \{\pm 1\}$ so that $u(t)=W^{\iota}_{(\lambda)}$ for $t\in I$.
  \item \label{caselr4} There exists $\ell \in \RR$, $\ell\neq 0$ such that if $t_0\in I$ and $R_0>0$ is such that \begin{equation}
\label{smallness}
 \|\vec{u}(t_0)\|_{\dot{H}^1_{R_{0}}\times L^2_{R_{0}}}=\varepsilon\leq \varepsilon_0,
\end{equation} 
   we have, for all $R\geq R_0$,
  \begin{equation} \label{nonradiative:quantitativebound}
\| \partial_t u(t_0)-\ell/r^4\|_{L^2_{R}}\leq C \eps^2\left(\frac{R_0}{R} \right)^{3/2}.
\end{equation}
\end{enumerate}
The constants $C$ and $\eps_0$ are independent of $u$.
\end{proposition}

 We will prove Proposition \ref{P:nonradiative}  as a consequence of the main result of \cite{CoDuKeMe22Pb}, which gives a complete classification of small non-radiative solutions of energy-critical wave equations (including \eqref{NLWabs} and \eqref{NLW}) outside wave cones. Let us mention however that the full strength of \cite{CoDuKeMe22Pb} is not needed here, and that only the conclusion of Proposition \ref{P:nonradiative} is necessary to prove the soliton resolution. See also Section 4 of \cite{CoDuKeMe22Pv1} for a self-contained proof. 
We will prove Proposition \ref{P:nonradiative} as a consequence of the following lemma:
\begin{lemma}
 \label{L:nonradiative}
 There exist constants $C>0$ and $\eps_0>0$ with the following property.
Let $t_0\in \RR$, $R_1\geq 0$ and $u$ be a $(t_0,R_1)$ non-radiative solution of \eqref{NLW} which is not a stationary solution. Then there exists $\ell \in \RR$, $\ell\neq 0$ such that if $R_0\geq R_1$ satisfies \eqref{smallness}, we have, for all $t\in \RR$, for all $R\geq R_0$, 
  \begin{equation} \label{nonradiative:quantitativebound'}
  \left\|\partial_t u(t)-\ell/r^4\right\|_{L^2_{R+|t|}}\leq C \eps^2\left(\frac{R_0}{R} \right)^{2}.
\end{equation}
 \end{lemma}
\begin{proof}[Proof of Proposition \ref{P:nonradiative} assuming the lemma]
 Let $u$ be as in the proposition and $t_0\in I$. Then by the lemma, we see that there exists $\ell=\ell(t_0)$ such that if \eqref{smallness} holds, then \eqref{nonradiative:quantitativebound'} is satisfied. Note that \eqref{nonradiative:quantitativebound'} at $t=t_0$  is exactly the desired bound \eqref{nonradiative:quantitativebound'}. We just need to check that $\ell$ is independent of $t_0$. 
 
 Since $\|r^{-4}\|_{L^2_R}=\frac{1}{R\sqrt{2}}$ we see that \eqref{nonradiative:quantitativebound'} implies that for all $t\in I$
 $$ \lim_{R\to\infty} R\|\partial_t u(t)\|_{L^2_R}=\ell/\sqrt{2}.$$
 Since the previous limit is independent of $t_0$, we obtain that 
 $\ell$ is also independent of $t_0$, concluding the proof.
\end{proof}

 \begin{proof}[Proof of Lemma \ref{L:nonradiative}]
 We use the main result of \cite{CoDuKeMe22Pb}, together with a symetrisation argument from \cite{DuKeMaMe21P}. Let $u$ be a $(0,R_1)$ non-radiative solution of \eqref{NLWabs} or \eqref{NLW}. Let $R_0\geq R_1$ such that \eqref{smallness} holds. For $\cbf=(c_0,c_1)\in \RR^2$, we will denote 
 $$ |\cbf|_R=|(c_0,c_1)|_R=\frac{|c_0|}{R^2}+\frac{|c_1|}{R},$$
 so that 
 \begin{equation}
  \label{NR10}
  \left\|\left( \frac{c_0}{r^4},\frac{c_1}{r^4} \right)\right\|_{\HHH_R}\approx |(c_0,c_1)|_R.
 \end{equation} 
 We assume without loss of generality $t_0=0$ to lighten notation.
 
 According to Theorem 1.2 of \cite{CoDuKeMe22Pb}, taking $\eps_0$ small enough, we have that for all $R\geq R_0$, there exists $\cbf(R)=(c_0(R),c_1(R))$ such that 
 \begin{gather}
  \label{NR11}
  \vec{u}(t,r)=\left(\frac{c_0(R)}{r^4},\frac{c_1(R)}{r^4}\right)+\vec{h}_R(t,r),\\
  \label{NR11''}\int_{R}^{\infty}\partial_r\left(\frac{1}{r^4}\right)\partial_r h_R(0,r)r^5dr=\int_{R}^{\infty}\frac{1}{r^4}\partial_th_R(0,r)r^5dr=0\\
\label{NR11'}
\forall t,\; \forall \widetilde{R}\geq R+|t|,\quad \left\|\vec{h}_R(t)\right\|_{\HHH_{\widetilde{R}}}\lesssim \frac{R}{\widetilde{R}}\left|\cbf(R)\right|^2_{R}.
 \end{gather} 
 We note that \eqref{smallness} implies $|\cbf(R)|_R\lesssim \eps$ for $R\geq R_0$. Thus by \eqref{NR11'} at $t=0$, $R=R_0$,
 \begin{equation*}
  \left\|h_{R_0}(0)\right\|_{\HHH_R}\lesssim \frac{R_0}{R}\eps^2.
 \end{equation*} 
 Combining with \eqref{NR11}, again at $t=0$, $R=R_0$, we obtain, for $R\geq R_0$
 \begin{equation}
  \label{est_u}
  \|(u_0,u_1)\|_{\HHH_R}\lesssim |\cbf(R_0)|_R+\|\vec{h}_{R_0}(0)\|_{\HHH_R}\lesssim \frac{R_0}{R}\eps+\frac{R_0}{R}\eps^2\lesssim \frac{R_0}{R}\eps.
 \end{equation} 
 Hence using the orthogonality \eqref{NR11''},
\begin{equation}
 \label{NR13}
 |\cbf(R)|_{R}\leq \frac{R_0}{R} \eps.
\end{equation}
Combining \eqref{est_u} with small data theory, we also obtain
\begin{equation}
\label{NR14}
\left\|\indic_{\{|x|>R+|t|\}}u\right\|_{L^2L^4}+\sup_{t\in \RR}\|u(t,r)\|_{\HHH_{R+|t|}}\lesssim \frac{R_0}{R}\eps. 
\end{equation} 
We let 
$u_{\pm}(t)=\frac{u(t)\pm u(-t)}{2}.$
Then 
\begin{equation}
 \label{equ-}
 \partial_t^2u_--\Delta u_-=\frac{1}{2}\left(F(u(t))-F(u(-t))\right),\quad \vec{u}_-(0)=(0,\partial_tu(0)),
\end{equation} 
where $F(u)=|u|u$ or $F(u)=u^2$. Since $\left|F(u(t))-F(u(-t))\right|\lesssim |u_+(t)u_-(t)|$, we deduce, using Strichartz estimates, H\"older and \eqref{NR14} that for $R\geq R_0$,
\begin{equation}
 \label{NRR20}
 \left\|u_-\indic_{\{|x|>R+|t|\}}\right\|_{L^2L^4}\lesssim \left\|\partial_tu(0)\right\|_{L^2_R}.
\end{equation} 
By the channel energy bound for odd solutions (see Proposition \ref{pr:channels}), \eqref{NR14} and \eqref{NRR20}
$$\left\|\partial_th_R(0)\right\|_{L^2_R} \lesssim \left\|u_+u_-\indic_{\{|x|\geq R+|t|\}}\right\|_{L^1L^2}\lesssim \frac{R_0}{R} \eps\left\| \frac{c_1(R)}{r^4}+\partial_th_R(0)\right\|_{L^2_R}.$$
Thus
\begin{equation}
 \label{NRR21}
 \left\|\partial_th_R(0)\right\|_{L^2_R}\lesssim \frac{R_0}{R} \eps\left\|\frac{c_1(R)}{r^4}\right\|_{L^2_R} =\frac{R_0}{R} \eps \frac{|c_1(R)|}{R^2}.
\end{equation} 
Let $R_0\leq R\leq R'\leq 2R$. By \eqref{NR11} at $R$ and $R'$ and \eqref{NRR21},
$$ \frac{1}{R^2}\left|c_1(R)-c_1(R')\right|\approx \left\| \frac{c_1(R)-c_1(R')}{r^4}\right\|_{L^2_{R'}}\lesssim \frac{R_0}{R} \times \frac{\eps |c_1(R)|}{R^2}.$$
That is
\begin{equation}
 \label{NRR22}
 \left|c_1(R)-c_1(R')\right|\lesssim \frac{R_0}{R}  \eps|c_1(R)|.
\end{equation} 

\noindent\textbf{Case 1.} There exists $R\geq R_0$ such that $c_1(R)=0$. Then by \eqref{NRR21}, $\partial_th_R(0,r)=0$ for $r\geq R$. Thus $\partial_tu(0,r)=0$ for $r\geq R$. 
Choosing a stationary solution 
$$ Z\in \Big\{0\Big\}\cup\Big\{W_{(\lambda)}^{\iota}, \;\iota\in \{\pm\},\lambda>0\Big\} $$
such that $Z$ has the same orthogonal projection on the space spanned by $\frac{1}{r^4}$ as $u_0(r)$, we see by Theorem 1.2 of \cite{CoDuKeMe22Pb} that we must have 
$ u_0(r)=Z(r)$
concluding the proof in this case.

\medskip

\noindent\textbf{Case 2.} For all $R\geq R_0$, $c_1(R)\neq 0$. Thus $c_1(R)$ has constant sign, say $c_1(R)>0$ for all $R\geq R_0$.  By \eqref{NRR22}, 
\begin{equation}
 \label{NRR23}
 \left|\frac{c_1(R')}{c_1(R)}-1\right|\lesssim \frac{R_0}{R} \eps, \quad R_0\leq R\leq R'\leq 2R.
\end{equation} 
In particular,
\begin{equation} 
 \label{NRR23'}
 \left|\frac{c_1(2^{k+1}R)}{c_1(2^kR)}-1\right|\lesssim \frac{1}{2^k}\frac{R_0}{R} \eps, \quad R_0\leq R, \; k\in \NN.
\end{equation} 
Fixing $R\geq R_0$ and letting $a_k(R)=\log\left(c_1(2^{k+1}R)\right)-\log\left(c_1(2^{k}R)\right)$, we see by \eqref{NRR23'} that $|a_k(R)|\lesssim \frac{1}{2^k}\frac{R_0}{R} \eps$. Thus $\sum_{k\geq 0} |a_k|\lesssim \frac{R_0}{R} \eps$. As a consequence, $\log(c_1(2^kR))$ has a limit $L=L(R)\in \RR$ such that 
\begin{equation}
 \label{NRR25}
\left|\log(c_1(R))-L\right|\lesssim \frac{R_0}{R}\eps.
 \end{equation} 
Of course $L(R)=L(2R)$, and by \eqref{NRR23}, \eqref{NRR25}, we conclude that $L$ is independent of $R$, and is the limit of $c_1(R)$ as $R\to\infty$. Letting $\ell=e^L>0$, we obtain 
$\left|\log\left( \frac{c_1(R)}{\ell} \right)\right|\lesssim \frac{R_0}{R} \eps$, which yields
$\left|\frac{c_1(R)}{\ell} -1\right|\lesssim \frac{R_0}{R} \eps$. In particular (letting $R=R_0$), $|\ell|\lesssim R_0\eps$, and thus
\begin{equation}
 \label{NR21}
 \left|c_1(R)-\ell\right|\lesssim \frac{R_0^{2}}{R}\eps^2,\quad R\geq R_0,
\end{equation} 
which yields
 \begin{equation}
  \label{NR32}
  \left\|\frac{c_1(R)-\ell}{r^4}\right\|_{L^2_R}\lesssim \eps^2\left( \frac{R_0}{R} \right)^{2}.
 \end{equation} 
Combining \eqref{NR11}, \eqref{NR11'}, \eqref{NR13} and \eqref{NR32}, we obtain \eqref{nonradiative:quantitativebound'}. 
 \end{proof}

%
%
%

In the next subsection, we will use Proposition \ref{P:nonradiative} to obtain a lower bound of the exterior scaling parameter for a non-radiative solution which is close to a multisoliton. This lower bound is crucial in the proof of Theorem \ref{T:main}, in Section \ref{S:proof_resolution}.

\subsubsection{Non-radiative solution close to a multisoliton}
\label{Sub:close_multisoliton}

In this subsection, we fix $J\geq 1$, and consider a radial solution $u$ of \eqref{NLWabs} or \eqref{NLW}, defined on $\left\{(t,x)\in \RR^6\;:\;|x|>t\right\}$, which is $(0,0)$ \emph{non-radiative} and has initial data $(u_0,u_1)\in \HHH$. We assume that there exists $\lambdabf=(\lambda_j)^J$, with $0<\lambda_J<\ldots<\lambda_1$ and signs $(\iota_j)_{1\leq j\leq J}\in \{\pm 1\}^J$, with the convention that $(\iota_j)_{1\leq j \leq J}\equiv (1,...,1)$ for Equation \eqref{NLW}, such that:
\begin{gather}
\label{F160}
 \left\|\vec{u}(0)-(M,0)\right\|_{\HHH}=:\delta\leq \eps_J \ll 1\quad \text{ where }M=\sum_{j=1}^J \iota_j W_{(\lambda_j)}\\
\label{F161'}
\gamma\leq \eps_J\ll 1,
 \end{gather}
where as before $\gamma:=\gamma(\lambdabf)=\max_{1\leq j\leq J-1} \lambda_{j+1}/\lambda_j$. 
Denote 
$$h_0=u_0-M.$$
Using the implicit function theorem (see Lemma B.1 in \cite{DuKeMe19Pb}), we can change the scaling parameters $(\lambda_j)_j$ so that the following orthogonality relations hold:
\begin{equation}
 \label{F162}
 \forall j\in \llbracket 1,J\rrbracket,\quad \int \nabla_x h_0\nabla_x (\Lambda W)_{(\lambda_j)}=0.
\end{equation} 
We expand $u_1=\partial_t u(0)$ as follows:
\begin{equation}
 \label{F164}
 u_1=\sum_{j=1}^J\alpha_j\left(\Lambda W\right)_{[\lambda_j]}+g_1,\qquad  \forall j\in \llbracket 1,J\rrbracket \quad \int g_1 (\Lambda W)_{[\lambda_j]} =0,
\end{equation}
where by definition $f_{[\lambda]}(x)=\lambda^{-3}f(x/\lambda)$ for $f\in L^2(\RR^6)$. 
We first prove:
\begin{lemma}
\label{L:estimateg1}
\begin{equation}
\label{F180}
\|h_0\|_{Z_{-2,\lambdabf}}\lesssim\delta^{2}+\gamma^{2}|\log \gamma|, \qquad     \left\|g_1\right\|_{L^2}\lesssim\delta^{2}+\gamma^2.
\end{equation} 
  \end{lemma}
\begin{proof}
The proof being the same for the $u^2$ and $|u|u$ nonlinearities, we only give it for the $u^2$ nonlinearity (Equation \eqref{NLW}), to ease notations. We let $h(t)=u(t)-M$. Then
 \begin{equation*}
  \partial_t^2h-\Delta h=M^2+2Mh +h^2-\sum_{j=1}^J W_{(\lambda_j)}^2.
 \end{equation*}
Thus 
$$\partial_t^2h+L_{\lambdabf}h=h^2-2\sum_{j\neq k} W_{(\lambda_j)}W_{(\lambda_k)}.$$
By finite speed of propagation, $h$ coincides, for $|x|>|t|$, with the solution $\tilde{h}$ of 
\begin{equation}
 \label{eq:htilde}
 \left\{
 \begin{aligned}
  \partial_t^2\tilde{h}+L_{\lambdabf} \tilde{h} &=\Big( h^2-2\sum_{j\neq k}W_{(\lambda_j)}W_{(\lambda_k)} \Big)\indic_{\{|x|>|t|\}} \\
 \tilde{h}_{\restriction t=0}&= (h_0,u_1)
 \end{aligned}
 \right.
\end{equation}
We can thus rewrite the first line of \eqref{eq:htilde} as 
$$\partial_t^2\tilde{h}+L_{\lambdabf} \tilde{h} =\Big( \tilde{h}^2-2\sum_{j\neq k}W_{(\lambda_j)}W_{(\lambda_k)} \Big)\indic_{\{|x|>|t|\}}$$
Since $\|(h_0,u_1)\|_{\HHH}=\delta$ and by explicit computations (see \eqref{estim2} in the appendix), if $j\neq k$, 
$\left\|\indic_{\{|x|>|t|\}}W_{(\lambda_j)}W_{(\lambda_k)} \right\|_{L^1L^2}\lesssim \gamma^2|\log \gamma|$, we deduce, using a standard bootstrap argument and Strichartz estimates,
\be \label{bd:nonradiativemulti:tildeh} \big\|\tilde{h}\big\|_{L^2_tL^4_x}\lesssim \delta+\gamma^2|\log \gamma|.\ee
By Proposition \ref{pr:channels} using that the solution $u$ is $(0,0)$ non-radiative, we obtain
\begin{multline*}
\|\nabla \Pi_{\dot{H}^1,\lambdabf}^{\bot}h_0\|_{Z_{-3,\lambdabf}}
\lesssim (\delta+\gamma^2|\log\gamma|)^2+\gamma^2|\log\gamma|+\gamma \|(h_0,u_1)\|_{\HHH}\\
\lesssim \delta^2+\gamma^2|\log\gamma|+\gamma\delta\lesssim \delta^2+\gamma^2|\log\gamma|. 
\end{multline*}
The above estimate, the Sobolev embedding estimate $\| \Pi_{\dot{H}^1,\lambdabf}^{\bot}h_0\|_{Z_{-2,\lambdabf}}\lesssim \|\nabla \Pi_{\dot{H}^1,\lambdabf}^{\bot}h_0\|_{Z_{-3,\lambdabf}}$, and $h_0=\Pi_{\dot{H}^1,\lambdabf}^{\bot}h_0$   imply the first inequality in \eqref{F180}. We define the odd component $\tilde h_-(t)=\frac{1}{2}(\tilde h(t)-h(-t))$ that solves:
$$\partial_t^2\tilde{h}_-+L_{\lambdabf} \tilde{h}_- =\frac 12 \Big( \tilde{h}^2(t)-\tilde h^2(-t) \Big)\indic_{\{|x|>|t|\}}.$$
By Proposition \ref{pr:channels} again, using that the solution $u$ is $(0,0)$ non-radiative and \eqref{bd:nonradiativemulti:tildeh}, we obtain
\begin{multline*}
\|\Pi_{L^2,\lambdabf}^{\bot}\pa_t \tilde h_-(0)\|_{L^2}=\|g_1\|_{L^2}\\
\lesssim (\delta+\gamma^2|\log\gamma|)^2+\gamma \|u_1\|_{L^2}\lesssim \delta^2+\gamma^4|\log\gamma|^2+\gamma \delta\lesssim \delta^2+\gamma^2. 
\end{multline*}
This is the second inequality in \eqref{F180}.

\end{proof}

\begin{proposition}[Lower bound on the exterior scaling parameter]
\label{P:lower_bound}
There is a constant $C_0>0$ with the following property.
Let $u$ be as above. Assume furthermore that $u$ is not a stationary solution. Then if $\eps_J$ is small enough
$$ \lambda_1\geq \frac{\ell}{C_0\sqrt{\delta}},$$
where $\ell\neq 0$ is given by Proposition \ref{P:nonradiative}.
\end{proposition}
\begin{proof}
 Let $R_0=\lambda_1/\sqrt{\delta}$. Then 
 $$\|(u_0,u_1)\|_{\HHH_{R_0}}\lesssim\left\|(M,0)\right\|_{\HHH_{R_0}}+\delta$$
 and
 $$ 
 \|(M,0)\|_{\HHH_{R_0}}\leq \sum_{j=1}^{J}\big\|W_{(\lambda_j)}\big\|_{\dot{H}^1_{R_0}}=
\sum_{j=1}^{J}\left\|W\right\|_{\dot{H}^1_{R_0/\lambda_j}}\lesssim \delta.$$
 where we have used that $\|W\|_{\hdot_R}\approx R^{-2}$ for large $R$. As a consequence,
 $$ \|(u_0,u_1)\|_{\HHH_{R_0}}\lesssim \delta.$$
 Taking $\eps_J$ small enough we deduce, from Proposition \ref{P:nonradiative}
 $$C\delta \geq \|u_1\|_{L^2_{R_0}}\geq \| \ell r^{-4}\|_{L^2_{R_0}}-\|\partial_t\tilde{u}\|_{L^2_{R_0}}\geq \frac{\ell}{R_0}-C \delta^2\geq \frac{\ell \sqrt{\delta}}{\lambda_1}-C\delta^2.$$
 Taking a smaller $\eps_J$ if necessary, we obtain the conclusion of the proposition. 
\end{proof}

\section{Proof of the soliton resolution}
\label{S:proof_resolution}

In this Section, we prove Theorem \ref{T:main}. We first focus on the case $T_+(u)=+\infty$ for the $u^2$ nonlinearity (Equation \eqref{NLW}), and then treat the $|u|u$ nonlinearity (Equation \eqref{NLWabs}). The case $T_+(u)<+\infty$ can be treated similarly and we omit it. The proof follows the same lines as the proof in the odd-dimensional case (see \cite{DuKeMe19Pc}) and we will only detail the novelties.

\subsection{Setting of the proof}
\label{Sub:setting_proof}
Let $u$ be a solution of \eqref{NLW} such that $T_+(u)=+\infty$ and
\begin{equation}
 \label{R10}
 \limsup_{t\to+\infty} \|\vec{u}(t)\|_{\HHH}<\infty.
\end{equation} 
Let $v_L$ be the unique solution of the free wave equation $\partial^2_tv_L-\Delta v_L=0$ such that 
\begin{equation}
 \label{R11}
 \forall A\in \RR,\quad \lim_{t\to +\infty} \int_{|x|\geq A+|t|} |\nabla_{t,x}(u-v_L)(t,x)|^2\,dx=0
\end{equation}
(see \cite[Proposition 4.1]{Rodriguez16}, the proof there does not use that the dimension is odd and also works in even dimension).
For $J\geq 1$, $(f,g)\in \HHH$, we denote
\begin{equation}
\label{R12}
d_{J}(f,g)
=\inf_{\lambdabf \in \Lambda_J} \left\{\Big\|(f,g)-\sum_{j=1}^J (W_{(\lambda_j)},0)\Big\|_{\HHH}+\gamma(\lambdabf)\right\},
\end{equation} 
where $\Lambda_J=\{\lambdabf=(\lambda_1,\ldots,\lambda_J), \ 0<\lambda_J<\ldots<\lambda_2<\lambda_1 \}$, and as before:
$$
\gamma(\lambdabf)=\max_{2\leq j\leq J} \frac{\lambda_j}{\lambda_{j-1}} \in (0,1).
$$
Assume that $u$ does not scatter forward in time.
By \cite{JiaKenig17}, we know that there exists $J\geq 1$,  and a sequence $\{t_n\}_n\to+\infty$ such that
\begin{equation}
 \label{R13}
 \lim_{n\to\infty} d_{J}(\vec{u}(t_n)-\vec{v}_L(t_n))=0.
\end{equation} 
\begin{remark}
 The article \cite{JiaKenig17} treats the case of a nonlinearity of the form $|u|u$. However a slight modification of the argument yields \eqref{R13} for equation \eqref{NLW}. The following result is needed: 
 \begin{equation}
  \label{uniqueness_elliptic}
 -\Delta f=f^2,\; f\in \dot{H}^1_{\textrm{rad}}(\RR^6)\Longrightarrow f\equiv 0\text{ or }\exists \lambda>0,\; f \equiv W_{(\lambda)}. 
 \end{equation} 
 The classification of radial $\dot{H}^1$ solutions to $-\Delta f=|f|f$ on $\RR^6$ is well-known. To prove \eqref{uniqueness_elliptic}, it is thus sufficient to prove that any radial, $\dot{H}^1(\RR^6)$ of $-\Delta f=f^2$ is nonnegative. This follows from the fact that for such a solution, $\partial_r f$ is nonpositive. Indeed, $r^5 \partial_r f$ is nonincreasing (by the equation). Thus 
$$ \forall 0<r\leq r_0,\quad r^5 \partial_rf(r)\geq r_0^5\partial_rf(r_0),$$
and the fact that $\partial_r f\in L^2$ implies that $\partial_rf(r_0)$ cannot be positive. 
 \end{remark}
We will prove by contradiction that $\lim_{t\to\infty}d_J(\vec{u}(t)-\vec{v}_L(t))=0$. We thus assume that there exists a small $\eps_0>0$ and a sequence $\{\tilde{t}_n\}_n\to+\infty$ such that 
\begin{gather}
 \label{R14}
 \forall n,\quad \tilde{t}_n<t_n\\
 \label{R15}
 \forall n,\quad \forall t\in (\tilde{t}_n,t_n],\quad d_J(\vec{u}(t)-\vec{v}_L(t))<\eps_0\\
 \label{R16}
 d_J(\vec{u}(\tilde{t}_n)-\vec{v}_L(\tilde{t}_n))=\eps_0.
\end{gather}
We will denote 
$$U=u-v_L,\quad h(t)=u(t)-v_L(t)-M(t)=U(t)-M(t),\quad M(t)=\sum_{j=1}^J W_{(\lambda_j(t))}.$$
The implicit function theorem (see Lemma B.1 \cite{DuKeMe19Pb}) implies that for all $t\in [\tilde{t}_n,t_n]$, we can choose $\lambdabf(t)=(\lambda_1(t),\ldots,\lambda_J(t))\in \Lambda_J$ such that 
\begin{equation}
 \label{R17}
 \forall j\in \llbracket 1,J\rrbracket, \quad \int \nabla h(t)\cdot\nabla (\Lambda W)_{(\lambda_j(t))}=0,
 \end{equation}
 and, in view of Remark B.2 in \cite{DuKeMe19Pb},
 \begin{equation}
 \label{R18}
 \big\| \left(h(t),\partial_t U(t)\right)\big\|_{\HHH}+\gamma(\lambdabf) \approx d_J(\vec{u}(t)-v_L(t)).
\end{equation} 
In the sequel, we will denote

$$ \gamma(t)=\gamma(\lambdabf(t)), \quad \delta(t)=\sqrt{\|h(t)\|^2_{\hdot}+\|\partial_tU(t)\|^2_{L^2}}.$$

We will expand $\partial_tU=\partial_tu-\partial_tv_L$ as follows:
\begin{equation}
 \label{exp_dt}
 \partial_tU(t)=\sum_{j=1}^J \alpha_j(t) \Lambda W_{[\lambda_j(t)]}+g_1(t),
\end{equation}
where
\begin{equation}
 \label{exp_dt_ortho}
 \forall j\in \llbracket 1,J\rrbracket,\quad \int g_1(t)\Lambda W_{[\lambda_j(t)]}=0.
\end{equation} 
We also define:
\begin{equation}
 \label{R155}
 \beta_j(t)=-\int (\Lambda W)_{[\lambda_j(t)]}\partial_tU(t)\,dx.
\end{equation} 
\subsection{Expansion along a sequence of times and renormalisation}
\label{Sub:expansion}
Consider a sequence of times $\{s_n\}_n$ with $s_n\in [\tilde{t}_n,t_n]$ for all $n$. Extracting subsequences, we define a
partition of the interval $\llbracket 1,J\rrbracket$ as follows. We let $1=j_1<j_2<\ldots<j_{K+1}=J+1$, so that $\llbracket 1,J\rrbracket=\cup_{k=1}^K \llbracket j_k,j_{k+1}-1\rrbracket$, with
\begin{equation}
\label{RR3}
\forall k\in \llbracket 1,K-1\rrbracket,\quad
\lim_{n\to\infty} \frac{\lambda_{j_{k+1}}(s_n)}{\lambda_{j_k}(s_n)}=0.
\end{equation} 
and, 
\begin{equation}
 \label{defnuj}
 \forall k\in \llbracket 1,K\rrbracket, \; \forall j\in \llbracket j_k,j_{k+1}-1\rrbracket,\quad 
 \nu_j=\lim_{n\to\infty} \frac{\lambda_j(s_n)}{\lambda_{j_k}(s_n)}>0.
\end{equation} 
We note that $\nu_{j_k}=1$. We have (see Lemma 5.2 of \cite{DuKeMe19Pb}):
\begin{lemma}
\label{L:expansion}
Under the above assumptions, for all $k\in \llbracket 1, K\rrbracket$, there exists $(V_0^k,V_1^k)$ in $\HHH$ such that, denoting by $V^k$ the solution of \eqref{NLW} with initial data $(V_0^k,V_1^k)$, then $V^k$ is defined on $\{|x|>|t|\}$ and is $(0,0)$ non-radiative. Furthermore, letting $J^k=j_{k+1}-j_k$, and
$$ V_n^k(t,x)=\frac{1}{\lambda_{j_k}^2(s_n)} V^k\left( \frac{t}{\lambda_{j_k}(s_n)},\frac{x}{\lambda_{j_k}(s_n)} \right),$$
we have (extracting subsequences if necessary),
\begin{equation}
 \label{LRR1}
 \lim_{n\to\infty} \left\|\vec{u}(s_n)-\vv_L(s_n)-\sum_{k=1}^{K} \vec{V}_n^k(0)\right\|_{\HHH}=0
\end{equation} 
and
\begin{equation}
 \label{LRR2}
 d_{J^k}\left( V_0^k,V_1^k \right)\leq C\eps_0.
\end{equation} 
More precisely, after extraction,
\begin{equation}
 \label{LRR3}
 \left\{
\begin{aligned}
 V^k_0&=\sum_{j=j_k}^{j_{k+1}-1}W_{(\nu_j)} +\check{h}_0^k\\
 V^k_1&=\sum_{j=j_k}^{j_{k+1}-1} \check{\alpha}_j(\Lambda W)_{[\nu_j]} +\check{g}_1^k,
\end{aligned}
 \right.
\end{equation}
where 
\begin{align}
\label{LRR4}
 \check{h}_0^k&=\underset{n\to\infty}{\wlim}\,\lambda_{j_k}^{2}(s_n)h\left(s_n,
 \lambda_{j_k}(s_n) \cdot\right)\\
 \label{LRR5}
 \check{\alpha}_j&=\lim_{n\to\infty}\alpha_j(s_n)\\
\label{LRR6}
 \check{g}_1^k&=\underset{n\to\infty}{\wlim}\, \lambda_{j_k}^{3}(s_n)g_1\left(s_n,\lambda_{j_k}(s_n) \cdot\right),
\end{align}
the first weak limit taking place in $\dot H^1$ and the second one in $L^2$. Furthermore, we have 
\begin{equation}
 \label{LRR7}
 JE(W,0)=\sum_{k=1}^K E\left( \vec{V}^k(0) \right)
\end{equation} 
\end{lemma}
Note that the orthogonality conditions \eqref{R17} and \eqref{exp_dt_ortho} and the limits \eqref{defnuj}, \eqref{LRR4}, \eqref{LRR5} and \eqref{LRR6} imply the orthogonality conditions
\begin{align}
 \label{RR4}
 \forall j\in \llbracket j_k,j_{k+1}-1\rrbracket
 \quad \int \nabla \check{h}_0^{k}\cdot \nabla(\Lambda W)_{(\nu_j)}=
 \int \check{g}_1^{k}\cdot (\Lambda W)_{[\nu_j]}=
 0
\end{align}
Also, we have the following expansion for all time outside the wave cone (see Claim 5.3 in \cite{DuKeMe19Pb})
\begin{equation}
\label{R60bis}
u(s_n+\tau)=v_L(s_n+\tau)+\sum_{k=1}^K V_n^k(\tau)+r_n(\tau),
\end{equation} 
where
\begin{equation*}
 \lim_{n\to\infty} \sup_{\tau}\int_{|x|\geq |\tau|} |\nabla_{\tau,x}r_n|^2\,dx=0.
\end{equation*} 
\subsection{Estimates on $\lambda_j$ and $\beta_j$}
\label{Sub:estim_modulation}
In this section and the next one, we will prove:
\begin{proposition}
\label{P:modulation}
 Let $u$, $J$, $\lambda_j$, $\beta_j$ be as above. Then for $\eps_0$ small enough, and $n$ large, for all $t\in [\tilde{t}_n,t_n]$, and all $j\in \llbracket 1,J\rrbracket$
 \begin{gather}
  \label{smallness_betaj}
  \beta_j^2(t) \leq C\gamma^{2}(t)+o_n(1)\\
    \label{smallness_lambdaj'}
  \left|\lambda_j'(t)-\kappa_2\beta_j(t)\right|\leq C\gamma^{2}(t)+o_n(1)\\
  \label{derivative_betaj}
  \left|\lambda_j(t)\beta'_j(t)-\kappa_0\left(\left( \frac{\lambda_{j+1}(t)}{\lambda_j(t)}\right)^2-\left( \frac{\lambda_j(t)}{\lambda_{j-1}(t)} \right)^2\right)\right|\leq C\gamma^{3}(t)+o_n(1),
 \end{gather}
 where $o_n(1)\to 0$ uniformly for $t\in [t_n,\tilde{t}_n]$, as $n\to\infty$
The constants $\kappa_0,\kappa_2$ are explicit positive constants that are independent of the parameters. The constant $C$ depends only on $J$. 
 \end{proposition}

\begin{proof}

Proposition \ref{P:modulation} is proved in the forthcoming lemmas. The estimate \eqref{smallness_betaj} follows from \eqref{alphabeta}, $|\alpha_j|\lesssim \delta$ and \eqref{gamma_delta}. The estimate \eqref{smallness_lambdaj'} follows from \eqref{F200}, \eqref{alphabeta} and \eqref{gamma_delta}. Finally, \eqref{derivative_betaj} is proved in Lemma \ref{L:second_derivative}.

\end{proof}

\begin{lemma}
 There exists a constant $C>0$, depending only on $J$ such that 
 \begin{gather}
  \label{est.h}
  \forall t\in [\tilde{t}_n,t_n], \quad \|h(t)\|_{Z_{-2,\lambdabf(t)}}\lesssim \delta^2+\gamma^2|\log \gamma|+o_n(1)\\
  \label{est.g1}
   \forall t\in [\tilde{t}_n,t_n], \quad \|g_1(t)\|_{L^2}\lesssim \delta^2+\gamma^2+o_n(1)\\
  \label{alphabeta}
  \forall t\in [\tilde{t}_n,t_n],\quad \left|\beta_j+\alpha_j\|\Lambda W\|_{L^2}^2\right|\lesssim \gamma\delta.
 \end{gather}
\end{lemma}
\begin{proof}

\textbf{Step 1.} \emph{Proof of \eqref{est.h} and \eqref{est.g1}}. We adapt the proof of Lemma 5.4 in \cite{DuKeMe19Pb}. We argue by contradiction and assume that, up to extracting a subsequence, there exists $(s_n)_n$ with $s_n\in [t_n,\tilde t_n]$ such that for any $L>0$, an $\epsilon_1>0$ exists such that for all $n$ large:
\be \label{errorbd:bd:hypothesis}
\begin{array}{l l}\| h(s_n)\|_{Z_{-2,\lambdabf(s_n)}} \geq L[\delta^2(s_n)+\gamma^2(s_n)|\log \gamma(s_n)|]+\epsilon_1,\\
\mbox{or } \quad \| g_1(s_n)\|_{L^2} \geq L[\delta^2(s_n)+\gamma^2(s_n)]+\epsilon_1.
\end{array}
\ee
In the proof, $C>0$ denotes a generic constant that is independent of $L$. Using Lemma \ref{L:expansion}, there exist $K\leq J$ and for each $1\leq k\leq K$ a non-radiative solution $V^k$ with initial data $(V^k_0,V^k_1)$ given by:
\begin{align*}
& V^k_0=\sum_{j_k}^{j_{k+1}-1}W_{(\nu_j)}+\check h_0^k, \\
& V^k_1=\sum_{j_k}^{j_{k+1}-1}\check \alpha_j (\Lambda W_{(\nu_j)})_{[\nu_j]}+\check g_1^k,
\end{align*}
(the notation $j_k$, $\nu_j$, $\check h^k_0$, $\check g_1^k$ being introduced in lemma \ref{L:expansion}) with
\be \label{errorbd:id:inter3}
\check\alpha_j=\lim_{n\to \infty}\alpha_j(s_n),
\ee
such that:
\be \label{errorbd:id:inter2}
\lim_{n\to \infty} \| \vec u(s_n)-\vec v_L(s_n)-\sum_{k=1}^K \vec V_{(\lambda_{j_k}(s_n))}(0) \|_{\mathcal H}=0.
\ee
We introduce (with the convention that $\gamma_k=0$ if $j_{k+1}=j_k+1$)
$$
\delta_k^2=\| \check h^k_0\|_{\dot H^1}+\| \pa_t V^k(0)\|_{L^2}^2, \qquad \gamma_k=\max_{j_k\leq j\leq j_{k+1}-1}\frac{\nu_{j+1}}{\nu_{j}}.
$$
Then, notice that, using \eqref{errorbd:id:inter2} and $\lambda_{j_{k+1}}/\lambda_{j_{k}}\to 0$ for all $k=1,...,K-1$,
\be \label{errorbd:id:inter1}
\delta^2(s_n)\rightarrow \sum_{k=1}^K \delta_k^2=:\delta_\infty \qquad \mbox{and} \qquad \gamma(s_n)\to \max_{1\leq k \leq K}\gamma_k=:\gamma_\infty.
\ee
As $V^k$ is non-radiative, applying Lemma \ref{L:estimateg1} and then using \eqref{errorbd:id:inter1} we obtain
\be \label{errorbd:id:inter5}
\| \check h_0^k \|_{Z_{-2,\nubf_k}}\leq C (\delta^2_\infty+\gamma_\infty^2|\log \gamma_\infty|) \quad \mbox{and} \quad \| \check g_1^k\|_{L^2}\leq C (\delta^2_\infty+\gamma_\infty^2),
\ee
where $\nubf_k=(\nu_{j_k},...,\nu_{j_{k+1}-1})$. We next remark that \eqref{errorbd:id:inter2} and \eqref{errorbd:id:inter3} imply:
\be \label{errorbd:id:inter6}
 \| h(s_n)-\sum_{k=1}^K \check h^k_{(\lambda_{j_k}(s_n))}(0)\|_{\dot H^1}+ \| g_1(s_n)-\sum_{k=1}^K \check g^k_{[\lambda_{j_k}(s_n)]}(0)\|_{L^2}\to 0.
\ee
We claim that for any $k=1,...,K$:
\be \label{errorbd:id:inter4}
\| \check h^k_{(\lambda_{j_k}(s_n))}(0)\|_{Z_{-2,\lambdabf(s_n)}}\leq C(\delta^2_\infty+\gamma^2_\infty|\log \gamma_\infty|)+o_n(1).
\ee
Then, combining \eqref{errorbd:id:inter6}, \eqref{errorbd:id:inter5} and \eqref{errorbd:id:inter4} shows:
$$
 \| h(s_n)\|_{Z_{-2,\lambdabf(s_n)}}\leq C(\delta^2_\infty+\gamma^2_\infty |\log \gamma_\infty|)+o_n(1) \; \mbox{and} \;\|  g_1(s_n)\|_{L^2}\leq C (\delta^2_\infty+\gamma_\infty^2)+o_n(1),
$$
contradicting \eqref{errorbd:bd:hypothesis} and \eqref{errorbd:id:inter1} for large enough $L$ and $n$. Hence the bounds \eqref{est.h} and \eqref{est.g1} of the lemma.

It then remains to show \eqref{errorbd:id:inter4}. We introduce $R_{n}^k= \sqrt{\lambda_{j_{k+1}-1}(s_n)\lambda_{j_{k+1}}(s_n)} $ for $k=1,...,K-1$, $R^0_{n}=\infty $ and $R^K_{n}=0$, and decompose:
$$
\check h^k_{(\lambda_{j_k}(s_n))}(0)= \indic_{\{R_{n}^k\leq |x|\leq R_{n}^{k-1}\}}h^k_{(\lambda_{j_k}(s_n))}(0)+\left(\indic_{\{|x|\leq R_{n}^k\}}+\indic_{\{|x|\geq R_{n}^{k-1}\}}\right) h^k_{(\lambda_{j_k}(s_n))}(0).
$$
Since $\check h^k_0\in \dot H^1$, we have $\check h_0^k(x)=o(|x|^{-2})$ as $|x|\to 0 $ and $|x|\to \infty$ by the radial Sobolev embedding. Since $R^k_n/\lambda_{j_{k}(s_n)}\to 0$ and $R^{k-1}_n/\lambda_{j_{k}(s_n)}\to \infty$ as $n\to \infty$, this implies:
$$
\left\| \Big(\indic_{\{|x|\leq R_{n}^k\}}+\indic_{\{|x|\geq R_{n}^{k-1}\}} \Big) h^k_{(\lambda_{j_k}(s_n))}(0)\right\|_{Z_{-2,\lambdabf(s_n)}}\to 0
$$
as $n\to \infty$, by the definition of the $\| \cdot \|_{Z_{-2,\lambdabf}}$ norm. Still by definition of the $\| \cdot \|_{Z_{-2,\lambdabf}}$ norm:
\begin{align*}
& \left\| \indic_{\{R_{n}^k\leq |x|\leq R_{n}^{k-1}\}}h^k_{(\lambda_{j_k}(s_n))}(0)\right\|_{Z_{-2,\lambdabf(s_n)}}  \ = \ \left\| \indic_{\left\{\frac{R_{n}^k}{\lambda_{j_k}(s_n)}\leq |x|\leq \frac{R_{n}^{k-1}}{\lambda_{j_k}(s_n)}\right\}} h^k_0\right\|_{Z_{-2,\frac{\lambdabf_k(s_n)}{\lambda_{j_k}(s_n)}}}\\
&\qquad \leq \| h^k_0\|_{Z_{-2,\frac{\lambdabf_k(s_n)}{\lambda_{j_k}(s_n)}}} \sim_{n\to \infty}\| h^k_0\|_{Z_{-2,\nubf_k}}
\end{align*}
where we wrote $\frac{\lambdabf_k(s_n)}{\lambda_{j_k}(s_n)}=(\frac{\lambda_j(s_n)}{\lambda_{j_k}(s_n)})_{j_k\leq j \leq j_{k+1}-1}$ and used $\frac{\lambda_j(s_n)}{\lambda_{j_k}(s_n)}\to \nu_j$. Combining the two above inequalities and \eqref{errorbd:id:inter5} shows the desired claim \eqref{errorbd:id:inter4} and ends Step 1.\\
  
\noindent \textbf{Step 2}. \emph{Proof of \eqref{alphabeta}}. We write 
 \begin{multline*}
  \beta_j(t)=-\int (\Lambda W)_{[\lambda_j]}\partial_tU\\
  =- \underbrace{\int (\Lambda W)_{[\lambda_j]} g_1}_{=0}-\alpha_j \|\Lambda W\|^2_{L^2}-\sum_{k\neq j}\alpha_k \int (\Lambda W)_{[\lambda_j]}(\Lambda W)_{[\lambda_k]}.
 \end{multline*}
 Since $|\alpha_k|\lesssim \delta$ and 
$\left|\int (\Lambda W)_{[\lambda_j]}(\Lambda W)_{[\lambda_k]}\right|\lesssim \gamma$ for $j\neq k$, \eqref{alphabeta} follows.
 \end{proof}
 We next prove, using the expansion of the energy:
 \begin{lemma}
 \label{L:energy}
 We have $J\geq 2$, and
  \begin{gather}
  \label{bound_delta2}
   \left|\frac{1}{2}\delta^2-\kappa_1\sum_{j=1}^{J-1}\left( \frac{\lambda_{j+1}}{\lambda_j} \right)^2\right|\lesssim o_n(1)+\gamma^{3},\\
   \label{gamma_delta}
   \gamma\approx \delta +o_n(1).
  \end{gather}
for some absolute constant $\kappa_1>0$.
 \end{lemma}
\begin{proof}
 Recall that 
 $$\lim_{t\to\infty} E(\vec{u}(t)-\vec{v}_L(t))=JE(W,0).$$
Expanding the energy
$$E(\vec{u}-\vec{v}_L)=E\left( \sum_{j=1}^J W_{(\lambda_j)}+h,\partial_t(u-v_L) \right),$$
we obtain
\begin{multline}
 \label{En10} 
 \bigg|\frac{1}{2}\|\partial_t(u-v_L)\|^2_{L^2}+\frac{1}{2}\|\nabla h\|^2_{L^2}\\
 +\frac{J}{2}\|\nabla W\|^2_{L^2}-\frac{J}{3}\|W\|^3_{L^3} -JE(W,0)+\sum_{1\leq j\leq J} \int \nabla W_{(\lambda_j)}\cdot \nabla h-\sum_{j=1}^{J}W_{(\lambda_j)}^2h\\
 +\sum_{1\leq j<k\leq J} \int \nabla W_{(\lambda_j)}\cdot\nabla W_{(\lambda_k)}-\sum_{\substack{1\leq j,k\leq J\\j\neq k}}  W_{(\lambda_j)}^2W_{(\lambda_k)} 
 \bigg| \\
 \lesssim \|h\|^3_{L^3}+\sum_{j=1}^J \int h^2 W_{(\lambda_j)}+\sum_{j\neq k} W_{(\lambda_j)}W_{(\lambda_k)} |h|+o_n(1)
\end{multline}
The first line of \eqref{En10} is exactly $\frac 12 \delta^2(t)$.  
 
The second line of \eqref{En10} equals to $0$, by the definition of the energy and the equation satisfied by $W$. 

Noting that for all $j,k$, we have
$ \int\nabla W_{(\lambda_j)}\cdot\nabla W_{(\lambda_k)}=\int W_{(\lambda_j)}^2W_{(\lambda_k)},$
we see that the third line of \eqref{En10} is equal to  
$-\sum_{1\leq j<k\leq J} W_{(\lambda_j)}^2W_{(\lambda_k)}.$
Furthermore, by direct computations (see the proof of Lemma 5.5 in \cite{DuKeMe19Pb}) for $1\leq j<k\leq J$:
$$\int W_{(\lambda_j)}^2W_{(\lambda_k)}=\left( \frac{\lambda_k}{\lambda_j} \right)^2\int \frac{(24)^2}{|x|^4} W^2dx+\OOO\left( \left(\frac{\lambda_k}{\lambda_j}\right)^4\left|\log \frac{\lambda_k}{\lambda_j} \right| \right).$$
Thus, introducing $\kappa_1=\int \frac{(24)^2}{|x|^4}W^2dx$, the third line of \eqref{En10} is equal to
$$ -\kappa_1\sum_{j=1}^{J-1}\left(\frac{\lambda_{j+1}}{\lambda_j} \right)^2+\OOO(\gamma^4|\log \gamma|).$$

We next consider the fourth line of \eqref{En10}. We have 
\begin{multline}
\label{En20}
  \int h^2W_{(\lambda_j)}dx=\sum_{k\in \ZZ} \int_{2^k\lambda_j}^{2^{k+1}\lambda_{j}} h^2(x)W_{(\lambda_j)}(x)dx\\
  \lesssim \sum_{k\in \ZZ} \left(\sup_{2^k\lambda_j\leq |x|\leq 2^{k+1}\lambda_j} W_{(\lambda_j)}(x)\right)\int_{2^k\lambda_j}^{2^{k+1}\lambda_j}h^2(x)dx\\
\lesssim \sum_{k\in \ZZ} \min\left( \frac{1}{\lambda_j^2},\frac{1}{2^{4k}\lambda_{j}^2} \right)\left( 2^{2k}\lambda_j^2 \right)\left\langle \log(2^k) \right\rangle^2\|h\|^2_{Z_{-2,\lambdabf}},
  \end{multline}
  where we have used 
  $$ 0\leq W_{(\lambda_j)}(x)\lesssim \min \left( \frac{1}{\lambda_j^2},\frac{\lambda_j^2}{|x|^4} \right)$$
  and the definition of $Z_{-2,\lambdabf}$. Thus
\begin{equation}
  \label{En30}
  \int h^2W_{(\lambda_j)}dx\lesssim \sum_{k\in \ZZ} \min\left( 2^{2k},2^{-2k} \right)\langle k\rangle ^2 \|h\|^2_{Z_{-2,\lambdabf}}\lesssim \gamma^4|\log \gamma|^2+\delta^4+o_n(1) 
\end{equation} 
  by \eqref{est.h}. In the case where $J=1$, the right-hand side is replaced by $\delta^4$. This implies $\delta(t)=o_n(1)$, a contradiction with the definition of $\tilde{t}_n$. Thus $J\geq 2$.
  
  Also, for $k<j$, using $0<W_{(\lambda_j)}(x)\lesssim \lambda_j^2 |x|^{-4}$ for $|x|\geq \sqrt{\lambda_j\lambda_k}$ and introducing $i_0=\lfloor \log_2 \sqrt{\frac{\lambda_j}{\lambda_k}}\rfloor$, a similar computation to \eqref{En20}-\eqref{En30} gives:
  \begin{multline*}
\int_{|x|\geq \sqrt{\lambda_j\lambda_k}} W_{(\lambda_j)}W_{(\lambda_k)}|h| \\
\lesssim \  \sum_{i\geq i_0} \sup_{2^k\lambda_j\leq |x|\leq 2^{k+1}\lambda_j} W_{(\lambda_k)}(x)\left( \int_{2^i\lambda_k}^{2^{i+1}\lambda_k} \frac{\lambda_j^4}{|x|^{8}}\right)^{\frac 12}\left( \int_{2^i\lambda_k}^{2^{i+1}\lambda_k} h^2\right)^{\frac 12}\\
\qquad \qquad  \lesssim  \lambda_j^2 \sum_{i\geq i_0} \min \left( \frac{1}{\lambda_k^2},\frac{1}{2^{4k}\lambda_k^2}\right) 2^{-2k}\lambda^{-2}_k 2^{2k}\lambda^{2}_k \langle \log 2^k \rangle^2 \| h\|_{Z_{-2,\lambdabf}} \\
 \qquad \qquad \qquad \lesssim \frac{\lambda_j^2}{\lambda_k^2}\left|\log \frac{\lambda_j}{\lambda_k}\right| \| h\|_{Z_{-2,\lambdabf}}  \ \lesssim \ \gamma^2 |\log \gamma| \delta^2+\gamma^{4}|\log \gamma|^2+o_n(1).
      \end{multline*}
One obtains similarly $\int_{|x|\leq \sqrt{\lambda_j\lambda_k}} W_{(\lambda_j)}W_{(\lambda_k)}|h|\lesssim \gamma^2 |\log \gamma| \delta^2+\gamma^4|\log \gamma|^2+o_n(1)$ and hence:
  \be
  \label{En31}
\sum_{k\neq j} \int W_{(\lambda_j)}W_{(\lambda_k)}| h| \lesssim  \gamma^2 |\log \gamma| \delta^2+\gamma^4|\log \gamma|^2+o_n(1)\lesssim \delta^4+\gamma^4|\log \gamma|^2+o_n(1).
  \ee
Combining the estimates above, we obtain that the left-hand side of \eqref{bound_delta2} is bounded by $C(\gamma^4|\log \gamma|^2+\delta^3)+o_n(1)$. Since 
  $\sum_{j=1}^{J-1} \left(\frac{\lambda_{j+1}}{\lambda_j} \right)^{2}\approx \gamma^2,$
  the conclusion of the lemma follows.
\end{proof}

Thanks to Lemma \ref{L:energy}, the estimates \eqref{est.h}, \eqref{est.g1} and \eqref{alphabeta} give:
 \be
  \label{est.g12}
  \forall t\in [\tilde{t}_n,t_n], \;\|h(t)\|_{Z_{-2,\lambdabf(t)}}\lesssim \gamma^2|\log \gamma|+o_n(1) \;\mbox{and} \; \|g_1(t)\|_{L^2}\lesssim \gamma^{2}+o_n(1).
 \ee

\subsection{System of equations and estimates on the derivatives}
\label{Sub:derivatives}
Under the above assumptions, using that 
$h(t)=U(t)-\sum_{j=1}^JW_{(\lambda_j)}=U(t)-M(t)$ and expanding the nonlinear wave equation \eqref{NLW}, we see that $(h(t),\partial_tU(t))$ satisfy the following system of equations for $t\in [\tilde{t}_n,t_n]$,
 \begin{equation}
  \label{R130}
  \left\{
  \begin{aligned}
   \frac{\partial h}{\partial t}&=\frac{\partial U}{\partial t}+\sum_{j=1}^J \lambda_j'(t)\left( \Lambda W\right)_{[\lambda_j(t)]}\\
  \frac{\partial}{\partial t}\left( \frac{\partial U}{\partial t} \right)&=\Delta h +2Mh+2Mv_L+(h+v_L)^2+M^2+\Delta M.
  \end{aligned}\right.
 \end{equation}
We estimate $\lambda_j'(t)$, using the orthogonality condition \eqref{R17} and the first equation in \eqref{R130}:
\begin{lemma}[Derivative of the scaling parameters] One has
\label{L:derivative}
 \begin{equation} 
  \label{F200}
  \left|\lambda'_j+\alpha_j\right|\lesssim \gamma^{2}+o_n(1),
 \end{equation} 
 where as before $o_n(1)$ goes to $0$ as $n\to\infty$, uniformly with respect to $t\in [\tilde{t}_n,t_n]$. 
 \end{lemma}
 \begin{proof}
 According to \eqref{R17},
 $$ \forall t\in I,\quad \int h(t)\frac{1}{\lambda_j^{3}}\left(\Delta \Lambda W\right)\left( \frac{x}{\lambda_j(t)} \right)\,dx=0.$$
 Differentiating with respect to $t$ and using the first equation in \eqref{R130}, we obtain
 \begin{multline*}
  0= \int \frac{\partial U}{\partial t}\frac{1}{\lambda_j^{3}}\left( \Delta \Lambda W \right)\left( \frac{x}{\lambda_j} \right)\,dx\\+\sum_{k=1}^J \lambda_k'\int \frac{1}{\lambda_k^{3}}(\Lambda W)\left( \frac{x}{\lambda_k} \right)\frac{1}{\lambda_j^{3}}(\Delta \Lambda W)\left( \frac{x}{\lambda_j} \right)\,dx
  \\
  -3\lambda_j'\int h \frac{1}{\lambda_j^{4}} \left( \Lambda_0\Delta \Lambda W\right)\left( \frac{x}{\lambda_j} \right)\,dx.
 \end{multline*}
 where $\Lambda_0=3+x\cdot \nabla$.

 The definition \eqref{exp_dt} of $g_1$ gives $\partial_tU=\sum_k \alpha_k(\Lambda W)_{[\lambda_k]}+g_1$. By the estimate \eqref{est.g12} on $g_1$, $|\alpha_j|\lesssim \delta \lesssim \gamma+o_n(1)$ for $1\leq j \leq J$ from \eqref{gamma_delta}, and the estimate
$$   \left|\int (\Lambda W)_{[\lambda_j]}\left(\Delta \Lambda W\right)_{[\lambda_k]}\right|\lesssim \gamma,\quad j\neq k,$$
that follows from direct computations, we obtain
\begin{equation*}
 \int \frac{\partial U}{\partial t} \frac{1}{\lambda_j^{3}} \left(\Delta \Lambda W\right)\left( \frac{x}{\lambda_j} \right)
 =-\alpha_j\|\Lambda W\|^2_{\hdot} +O\left( \gamma^{2}\right)+o_n(1).
\end{equation*}
Since $\|h\|_{\hdot}\lesssim \delta\lesssim \gamma+o_n(1)$ by the definition of $\delta$ and Lemma \ref{L:expansion}
$$\left|\lambda_j'\int h \frac{1}{\lambda_j^{4}} \left( \Lambda_0\Delta \Lambda W\right)\left( \frac{x}{\lambda_j} \right)\,dx\right|\lesssim |\lambda'_j|\|\nabla h\|_{L^2}
\lesssim \left(\gamma+o_n(1)\right)\,|\lambda_j'|.$$
Combining, we obtain
$$ \forall j,\quad \Big|\alpha_j\|\Lambda W\|^2_{\hdot}+\lambda'_j\|\Lambda W\|^2_{\hdot}\Big|\lesssim \gamma^{2} +\gamma \sum_{k} \left|\lambda_k'\right|+o_n(1),$$
and thus, letting $\alphabf=(\alpha_1,\ldots,\alpha_J)$,
$$\left| \lambdabf'+\alphabf\right|\lesssim |\lambdabf'|\gamma+\gamma^{2}+o_n(1).$$
This implies, recalling that $|\alpha|\lesssim \delta\lesssim \gamma+o_n(1)$,
$$ |\lambdabf'|\lesssim |\alphabf|+\gamma^{2}+o_n(1)\lesssim \gamma+o_n(1).$$
The desired estimate \eqref{F200} follows immediately from the two bounds above.
\end{proof}

\begin{lemma}[Second derivative of the scaling parameter] 
\label{L:second_derivative}
 For all $j\in \llbracket 2,J-1\rrbracket$,
 \begin{equation}
  \label{F270}
  \left|\lambda_j\beta'_j +\kappa_0\left(\left( \frac{\lambda_{j+1}}{\lambda_j} \right)^{2}-\left( \frac{\lambda_j}{\lambda_{j-1}}\right)^{2}\right)\right|\lesssim \gamma^{3}+o_n(1),
 \end{equation} 
where $\kappa_0>0$ is an absolute constant. Furthermore,
 \begin{equation}
  \label{F270'}
  \left|\lambda_J\beta'_J -\kappa_0 \left( \frac{\lambda_J}{\lambda_{J-1}}\right)^{2}\right|+\left|\lambda_1\beta'_1+\kappa_0 \left( \frac{\lambda_2}{\lambda_{1}}\right)^{2}\right|\lesssim \gamma^{3}+o_n(1),
 \end{equation} 
\end{lemma}
\begin{proof}
Differentiating the definition \eqref{R155} of $\beta_j$, we obtain 
\begin{equation}
\label{lambdabeta}
\lambda_j\beta_j'(t)=\lambda_j' \int \left(\Lambda_0\Lambda W\right)_{[\lambda_j]}\partial_tU-\lambda_j \int \left( \Lambda W \right)_{[\lambda_j]}\partial_t^2U
\end{equation}
We first prove that the first term of the right-hand side is negligible. Using the expansion \eqref{exp_dt} of $\partial_tU$, we obtain
\begin{multline*}
\int \left(\Lambda_0\Lambda W\right)_{[\lambda_j]}\partial_tU
=\int (\Lambda_0\Lambda W)_{[\lambda_j]} g_1\\+\alpha_j \underbrace{\int (\Lambda_0\Lambda W)_{[\lambda_j]}\left( \Lambda W \right)_{[\lambda_j]}}_{=0}+\sum_{k\neq j} \int \alpha_k(\Lambda_0\Lambda W)_{[\lambda_j]} (\Lambda W)_{[\lambda_k]}.
\end{multline*}
Hence, by \eqref{est.g12}, \eqref{F200}, $|\alpha_j|\lesssim \delta\lesssim \gamma+o_n(1)$ and since by direct computations
$\left|\int (\Lambda_0\Lambda W)_{[\lambda_j]} (\Lambda W)_{[\lambda_k]}\right|\lesssim \gamma$ for $k\neq j$, we obtain
\begin{equation}
 \label{F211} \left|\lambda_j'\int (\Lambda_0\Lambda W)_{[\lambda_j]}\partial_tU\right|\lesssim \gamma^{3}+o_n(1).
\end{equation} 
We next investigate the first term of the right-hand side of \eqref{lambdabeta}.
By the second equation in \eqref{R130}, we have
\begin{align}
\label{main_beta'1}
 \lambda_j\int \left( \Lambda W\right)_{[\lambda_j]}\partial_t^2U
=&-\int\left( \Lambda W \right)_{(\lambda_j)}L_{W_{(\lambda_j)}}h\\
\label{main_beta'2}
&+2\int \sum_{j\neq k} (\Lambda W)_{(\lambda_j)} W_{(\lambda_k)}h+\int (\Lambda W)_{(\lambda_j)}h^2\\
\label{main_beta'3}
&+\int (\Lambda W)_{(\lambda_j)} (2Mv_L+v_L^2+2v_Lh)\\
\label{main_beta'4}
&+ \int(\Lambda W)_{(\lambda_j)}(M^2+\Delta M).
 \end{align}
 where $L_{W_{(\lambda_j)}}=-\Delta-2 W_{(\lambda_j)}$.

 The term \eqref{main_beta'4} is estimated in \cite{DuKeMe19Pb}. Indeed,
 $$ \int (\Lambda W)_{(\lambda_j)}(M^2+\Delta M)=2\sum_{k<\ell} \int (\Lambda W)_{(\lambda_j)} W_{(\lambda_k)}W_{(\lambda_{\ell})}.$$
By direct computation (see \cite[(5.66) and (5.67)]{DuKeMe19Pb} where the computation is performed in any dimension), for some explicit constant $\kappa_0>0$:
$$2\int (\Lambda W)_{(\lambda_j)} W_{(\lambda_j)} W_{(\lambda_{j+1})}=\kappa_0\left( \frac{\lambda_{j+1}}{\lambda_{j}}\right)^2 +\OOO(\gamma^3)$$
and
$$2\int (\Lambda W)_{(\lambda_j)} W_{(\lambda_j)} W_{(\lambda_{j-1})}=-\kappa_0\left( \frac{\lambda_{j}}{\lambda_{j-1}}\right)^2 +\OOO(\gamma^3)$$
Also, if $k<\ell$ and $(k,\ell)\notin \{ (j,j+1),(j-1,j)\}$, we have
$$ \int (\Lambda W)_{(\lambda_j)}W_{(\lambda_k)}W_{(\lambda_{\ell})}=\OOO(\gamma^3).$$
We next prove that the other terms \eqref{main_beta'1},\eqref{main_beta'2} and \eqref{main_beta'3} are at most of order $\OOO(\gamma^{4}|\log \gamma|^2)+o_n(1)$, which will conclude the proof.

We have $\int (\Lambda W)_{(\lambda_j)}L_{W_{(\lambda_j)}}h=\int L_{W_{(\lambda_j)}}\big((\Lambda W)_{(\lambda_j)}\big)h=0$. By \eqref{En30}, \eqref{En31} and $\delta\lesssim \gamma+o_n(1)$ using \eqref{gamma_delta}:
$$\int (\Lambda W)_{(\lambda_j)}h^2=\OOO\left(\|h\|^2_{Z_{-2,\lambdabf}}\right)=\OOO(\gamma^4|\log \gamma|^2)+o_n(1)$$
and
$$\int (\Lambda W)_{(\lambda_j)}W_{(\lambda_k)}h=\OOO\left(\gamma^{4}|\log \gamma|^2\right)+o_n(1).$$
Thus \eqref{main_beta'2} is bounded, up to a constant, by $\gamma^{4}|\log \gamma|^2+o_n(1)$.

Finally, since $\lim_{t\to+\infty}\int |v_L(t)|^3=0$, we have $\eqref{main_beta'3}=o_n(1)$. This concludes the proof.
\end{proof}

\subsection{Restriction on the set of indices and end of the proof}
\label{subsection:solitonresolutionendproof}

We will next restrict the set of indices $\llbracket 1,J\rrbracket$ and the time interval $[\tilde{t}_n,t_n]$ so that estimates similar to \eqref{smallness_betaj}, \eqref{smallness_lambdaj'} and \eqref{derivative_betaj} hold without $o_n(1)$, and a lower bound of the exterior scaling parameter holds on the smaller time interval.

Recall that $J\geq 2$ by energy considerations (see Lemma \ref{L:energy}). Extracting subsequences if necessary, we let
$$ \left( \tU^j_0,\tU^j_1\right)=\wlim_{n\to\infty} \Big( \lambda_j^2(\tilde{t}_n)U\left( \tilde{t}_n,\lambda_j(\tilde{t}_n)\cdot \right), \lambda_j^3(\tilde{t}_n)U\left( \tilde{t}_n,\lambda_j(\tilde{t}_n)\cdot \right)\Big),$$
where $U=u-v_L$ as above, and the weak limit is in $\HHH$. We first note that for $1\leq j\leq J-1$,
\begin{equation}
 \label{limite_j}
\tilde{U}_0^j=W\Longrightarrow \lim_{n\to\infty} \frac{\lambda_{j+1}(\tilde{t}_n)}{\lambda_{j}(\tilde{t}_n)}=0.
\end{equation} 
Indeed, if (after extraction) $\lim_n \frac{\lambda_{j+1}(\tilde{t}_n)}{\lambda_j(\tilde{t}_n)}=\lambda_{\infty}>0$, then we see that $0<\lambda_\infty\lesssim \epsilon_0$ using \eqref{gamma_delta}. Then, writing $f_{\{\mu\}}(r)=\mu^{-4}f(\frac r\mu)$, by direct estimates we find
\begin{multline*}
\int W_{\{\lambda_\infty\}} \lambda_j^2(\tilde{t}_n)U( \tilde{t}_n,\lambda_j(\tilde{t}_n) 
\\= \int W_{\{\lambda_\infty\}} \left(W_{(\frac{\lambda_{j+1}(\tilde t_n)}{\lambda_j(\tilde t_n)})}+\sum_{i\neq j+1} W_{(\frac{\lambda_i(\tilde t_n)}{\lambda_j(\tilde t_n)})}+\lambda_j^2(\tilde{t}_n)h( \tilde{t}_n,\lambda_j(\tilde{t}_n)) \right)\\
\qquad = \int W^2 (1+o_n(1))+O(\gamma^2+\delta) \ \rightarrow  \int W^2+O(\gamma^2+\delta)
\end{multline*}
while $\int W_{\{\lambda_\infty\}} W=O(\lambda_\infty^2)=O(\epsilon^2_0)$, a contradiction, and \eqref{limite_j} follows.

Note that $\gamma(\tilde{t}_n)$ does not go to $0$ as $n$ goes to infinity, as this would imply $\lim_n\delta(\tilde{t}_n)=0$ by \eqref{gamma_delta}, a contradiction with the definition of $\tilde{t}_n$. By this and \eqref{limite_j}, there exists $\tJ\in \llbracket 1, J\rrbracket$ such that 
$$\forall j\in  \llbracket 1,\ldots,\tJ-1 \rrbracket, \quad (\tU^j_0,\tU^j_1)=(W,0),\quad 
(\tU^{\tJ}_0,\tU^{\tJ}_1)\neq (W,0).
$$
Furthermore, also by \eqref{limite_j}, one obtains $
\lim_{n\to\infty} \frac{\lambda_{j+1}(\tilde{t}_n)}{\lambda_{j}(\tilde{t}_n)}=0$ for $j\in \llbracket 1,\tJ-1\rrbracket$. In particular, we cannot have $\tJ=J$ (which would imply $\lim_n \gamma(\tilde{t}_n)=0$).

Using Lemma \ref{L:expansion}, we see that $\tU^{\tJ}$ is a $(0,0)$ non-radiative solution, and also that it cannot be of the form $\mu^2W(\mu\cdot)$ for some $\mu >0$. Thus by Proposition \ref{P:nonradiative}, there exists $\ell\neq 0$ such that if $R_0$ is chosen such that $\|(\tU_0^{\tJ},\tU_1^{\tJ})\|_{\HHH(R_0)}\leq \eps$, for some small $\eps>0$, then 
\begin{equation}
 \label{Re20}
 \forall R\geq R_0,\quad 
 \left\|\partial_t U^{\tJ}(t,r)-\frac{\ell}{r^4}\right\|_{L^2_R}\leq C\eps^2\left( \frac{R_0}{R} \right)^{5/4},
\end{equation} 
We let 
$$\tilde{\gamma}(t)=\sup_{\tJ\leq j\leq J-1} \frac{\lambda_{j+1}(t)}{\lambda_{j}(t)},\quad t_n'=\tilde{t}_n+T\lambda_{\tJ}(\tilde{t}_n),$$
where $T>0$ will be specified later. 
Then:
\begin{proposition}
\label{P:contra}
For $\eps_0$ small enough (independently of $T$), for any $T>0$, for $n$ large enough one has $t_n'< t_n$.
Moreover, for all $t\in [\tilde{t}_n,t_n']$, and all $j\in \llbracket \tJ,J\rrbracket$
 \begin{gather}
   \label{Re50}
  |\beta_j^2(t)|\lesssim \tilde \gamma^2(t) \\
  \label{Re21}
  \left|\lambda_j'(t)-\kappa_2\beta_j(t)\right|\leq C\tgamma^{2}(t)\\
  \label{Re22}
  \left|\lambda_j(t)\beta'_j(t)-\kappa_0\left(\left( \frac{\lambda_{j+1}(t)}{\lambda_j(t)}\right)^2-\left( \frac{\lambda_j(t)}{\lambda_{j-1}(t)} \right)^2\right)\right|\leq C\tgamma^{3}(t),\\
  \label{Re30}
  |\ell|\leq C\left( \frac{\lambda_{\tilde{J}}(t)}{\lambda_{\tJ}(\tilde{t}_n)} \right)\sqrt{\gamma(t)},
 \end{gather}
 where $\kappa_2=\| \lambda W \|_{L^2}^{-2}$.
\end{proposition}
In Appendix \ref{A:SDE}, we will prove, following \cite{DuKeMe19Pb}, that there is a constant $T^*$ (depending on $\ell$) such that if \eqref{Re21}, \eqref{Re22} and \eqref{Re30} hold on $[\tilde{t}_n,t_n']$, then $t_n'-\tilde{t}_n\leq T^*\lambda_{\tJ}(\tilde{t}_n)$. Thus Proposition \ref{P:contra} yields a contradiction if the parameter $T$ in the definition of $t_n'$ is chosen larger than $T^*$.

To conclude the proof of Theorem \ref{T:main}, we are left with proving Proposition \ref{P:contra}. The proof is the same as the proof of the corresponding result in \cite{DuKeMe19Pb} (see Subsections 6.2 and 6.3 there) and we only sketch it. We divide it into a few Lemmas.
\begin{lemma}
 \label{L:Re30}
 Let $\tau_n\in [\tilde{t}_n,t_n]$ be such that 
 \begin{equation}
  \label{Re31}
  \lim_{n\to\infty}\gamma(\tau_n)=0.
 \end{equation} 
 Then
 $$\lim_{n\to\infty} \frac{\tau_n-\tilde{t}_n}{\lambda_{\tJ}(\tilde{t}_n)}=+\infty.$$
\end{lemma}
Note that Lemma \ref{L:Re30} and the definitions of $t_n$ and $t_n'$ imply that $t_n'<t_n$ for large $n$.
\begin{proof}[Sketch of proof]
 We refer the reader to the proof of Lemma 6.3 and to Remark 6.5 in \cite{DuKeMe19Pb} for a detailed proof.  
 
 We argue by contradiction, assuming (after extraction) 
 \begin{equation}
  \label{Re32}
  \lim_{n\to\infty} \frac{\tau_n-\tilde{t}_n}{\lambda_{\tJ}(\tilde{t}_n)}=\sigma \in [0,\infty).
 \end{equation} 
 We use Lemma \ref{L:expansion} with $s_n=\tilde{t}_n$. With the notations of this Lemma, we see that for $k\in \llbracket 1,\tJ-1\rrbracket$, $V^k=W$ and that $V^{\tJ}=U^{\tJ}$. By the expansion \eqref{R60bis} and its time derivative at $\tau=\tau_n-\tilde{t}_n$, we obtain
 $$\vec{U}(\tau_n)=\sum_{k=1}^K \vec{V}_n^k(\tau_n-\tilde{t}_n)+\vec{r}_n(\tau_n),$$
 where $\lim_{n\to\infty} \left\|\vec{r}_n(\tau_n)\right\|_{\HHH(\{r>\tau_n-\tilde{t}_n\})}=0$. By \eqref{Re32}, 
 $$\partial_tV_n^{\tJ}(\tau_n-\tilde{t}_n,r)=\frac{1}{\lambda_{\tJ}^3(\tilde{t}_n)}\partial_tU^{\tJ}\left( \sigma,\frac{r}{\lambda_{\tJ}(\tilde{t}_n)} \right)+o_n(1)\text{ in }L^2\left( \{r>\tau_n-\tilde{t}_n\} \right).$$
 Since by \eqref{Re20},
 $$\lim_{n\to\infty}\int_{|x|>\tau_n-\tilde{t}_n}\frac{1}{\lambda_{\tJ}^{6}(\tilde{t}_n)}\left(\partial_t\tU^{\tJ}\left(\sigma,\frac{x}{\lambda_{\tJ}(\tilde{t}_n)}\right)\right)^2dx=\int_{|x|>\sigma}\left( \partial_t\tU^{\tJ}(\sigma,x) \right)^2dx>0,$$
 we obtain $\liminf_{n\to\infty}\left\|\partial_tU(\tau_n)\right\|_{L^2}>0$, a contradiction since \eqref{Re31} and Lemma \ref{L:energy} imply $\lim_{n\to\infty}\delta(\tau_n)=0$.
 \end{proof}
\begin{lemma}
 \label{L:Re40}
 \begin{gather}
  \label{Re40} \forall j\in \llbracket 1,\tJ-1\rrbracket,\quad \lim_{n\to\infty} \max_{\tilde{t}_n\leq t\leq t_n'} \left|\beta_j(t)\right|+\frac{\lambda_{j+1}(t)}{\lambda_j(t)}=0\\
  \label{Re41} \liminf_{n\to\infty} \min_{\tilde{t}_n\leq t\leq t_n'}\tilde{\gamma}(t)>0.
 \end{gather}
\end{lemma}
\begin{proof}[Sketch of proof]
(See Lemma 6.4 in \cite{DuKeMe19Pb}).
 The estimates \eqref{Re40} can be obtained by integrating the estimates on the derivatives of $\beta_j$ and $\lambda_j$ in Proposition \ref{P:modulation} and the fact that $t_n'-t_n=\lambda_{\tJ}(\tilde{t}_n)T\ll \lambda_j(\tilde{t}_n)$ for $j\in \llbracket 1,\tJ-1\rrbracket$.
 The limit \eqref{Re41} is a direct consequence of \eqref{Re40} and Lemma \ref{L:Re30}.
\end{proof}
\begin{lemma}
 \label{L:Re50}
 Let $\ell$ be defined by \eqref{Re20}. Then
 $$ \forall t\in [\tilde{t}_n,t_n'],\; |\ell|\leq 2C_0 \frac{\lambda_{\tJ}(t)\delta^{1/2}(t)}{\lambda_{\tJ}(\tilde{t}_n)},$$
 where $C_0$ is the constant in Proposition \ref{P:lower_bound}.
\end{lemma}
\begin{proof}[Sketch of proof]
 The proof is by contradiction, using Lemma \ref{L:expansion} and the expansion \eqref{R60bis} with $s_n=\tilde{t}_n$, $\tau\in [0,\tilde{t}_n-t_n']$ and the lower bound of the exterior scaling parameter in Proposition \ref{P:lower_bound}. We refer the reader to the proof of Lemma 6.6 in \cite{DuKeMe19Pb} for a detailed proof. 
\end{proof}
\begin{proof}[End of the Proof of Proposition \ref{P:contra}]
By \eqref{Re41}, the terms $o_n(1)$ in all the estimates of Subsections \ref{Sub:estim_modulation} and \ref{Sub:derivatives} can be bounded, for large $n$ and $t\in [\tilde{t}_n,t_n']$, by $\tilde{\gamma}(t)$. Also, by \eqref{Re40} and \eqref{Re41},  $\gamma(t)\leq 2\tilde{\gamma}(t)$ for large $n$, $t\in [\tilde{t}_n,t_n']$. Thus, we see that Proposition \ref{P:modulation} implies \eqref{Re21} and \eqref{Re22}. Since \eqref{Re30} is a direct consequence of Lemma \ref{L:Re50}, Proposition \ref{P:contra} follows.
\end{proof}

\subsection{Proof of the soliton resolution for the $|u|u$ nonlinearity}

Let $u$ be a solution of \eqref{NLWabs} such that $T_+(u)=+\infty$ and
\begin{equation}
 \label{R10abs}
 \limsup_{t\to+\infty} \|\vec{u}(t)\|_{\HHH}<\infty,
\end{equation} 
and let $v_L$ be the unique solution of the free wave equation $\partial^2_tv_L-\Delta v_L=0$ such that 
\begin{equation}
 \label{R11abs}
 \forall A\in \RR,\quad \lim_{t\to +\infty} \int_{|x|\geq A+|t|} |\nabla_{t,x}(u-v_L)(t,x)|^2\,dx=0
\end{equation}
(see \cite[Proposition 4.1]{Rodriguez16}). For $J\geq 1$ and $(f,g)\in \HHH$, we denote
\begin{equation}
\label{R12abs}
d_{J}(f,g)
=\inf_{\lambdabf \in \Lambda_J, \ (\iota_j)_{1\leq j \leq J}\in \{\pm 1\}^J } \left\{\Big\|(f,g)-\sum_{j=1}^J (\iota_j W_{(\lambda_j)},0)\Big\|_{\HHH}+\gamma(\lambdabf)\right\},
\end{equation} 
Assuming that $u$ does not scatter forward in time, by \cite{JiaKenig17}, we know that there exists $J\geq 1$, and a sequence $\{t_n\}_n\to+\infty$ such that
\begin{equation}
 \label{R13abs}
 \lim_{n\to\infty} d_{J}(\vec{u}(t_n)-\vec{v}_L(t_n))=0.
\end{equation} 
We again prove by contradiction that $\lim_{t\to\infty}d_J(\vec{u}(t)-\vec{v}_L(t))=0$. We thus assume that there exists a small $\eps_0>0$ and a sequence $\{\tilde{t}_n\}_n\to+\infty$ such that 
\begin{gather}
 \label{R14abs}
 \forall n,\quad \tilde{t}_n<t_n\\
 \label{R15abs}
 \forall n,\quad \forall t\in (\tilde{t}_n,t_n],\quad d_J(\vec{u}(t)-\vec{v}_L(t))<\eps_0\\
 \label{R16abs}
 d_J(\vec{u}(\tilde{t}_n)-\vec{v}_L(\tilde{t}_n))=\eps_0.
\end{gather}
The implicit function Theorem (Lemma B.1 \cite{DuKeMe19Pb}), as well as an extraction argument for the signs $(\iota_j(t))_{1\leq j\leq J}$, implies that for $\epsilon_0$ small enough, up to extracting a subsequence, there exist fixed signs $(\iota_j)_{1\leq j\leq J}$ such that for all $t\in [\tilde{t}_n,t_n]$, we can choose $\lambdabf(t)=(\lambda_1(t),\ldots,\lambda_J(t))\in \Lambda_J$ such that 
$$
U=u-v_L,\quad h(t)=u(t)-v_L(t)-M(t)=U(t)-M(t), \quad M(t)=\sum_{j=1}^J \iota_j W_{(\lambda_j(t))}
$$
with
\begin{equation}
 \label{R17abs}
 \forall j\in \llbracket 1,J\rrbracket, \quad \int \nabla h(t)\cdot\nabla (\Lambda W)_{(\lambda_j(t))}=0,
 \end{equation}
and (see Remark B.2 in \cite{DuKeMe19Pb}),
 \begin{equation}
 \label{R18abs}
 \big\| \left(h(t),\partial_t U(t)\right)\big\|_{\HHH}+\gamma(\lambdabf) \approx d_J(\vec{u}(t)-v_L(t)).
\end{equation} 
Thus, the only difference between the $u^2$ nonlinearity (Equation \eqref{NLW}) and the $|u|u$ nonlinearity (Equation \eqref{NLWabs}), is the appearance of fixed signs $(\iota_j)_{1\leq j \leq J}$ in the definition of the multisoliton $M$. The rest of the proof for the $u^2$ made in Subsections \ref{Sub:expansion}, \ref{Sub:estim_modulation}, \ref{Sub:derivatives} and \ref{subsection:solitonresolutionendproof} then extends readily, and we obtain the following for any fixed $T>0$, defining
$$
\tilde{\gamma}(t)=\sup_{\tJ\leq j\leq J-1} \frac{\lambda_{j+1}(t)}{\lambda_{j}(t)},\quad t_n'=\tilde{t}_n+T\lambda_{\tJ}(\tilde{t}_n).
$$
and referring to Subsection \ref{subsection:solitonresolutionendproof} for the notations:

\begin{proposition}
\label{P:contraabs}
For $\eps_0$ small enough (independently of $T$), for any $T>0$, for $n$ large enough one has $t_n'< t_n$.
Moreover, for all $t\in [\tilde{t}_n,t_n']$, and all $j\in \llbracket \tJ,J\rrbracket$
 \begin{gather}
   \label{Re50abs}
  |\beta_j^2(t)|\lesssim \tilde \gamma^2(t) \\
  \label{Re21abs}
  \left|\lambda_j'(t)-\kappa_2\beta_j(t)\right|\leq C\tgamma^{2}(t)\\
  \label{Re22abs}
  \left|\lambda_j(t)\beta'_j(t)-\kappa_0\left(\iota_j \iota_{j+1}\left( \frac{\lambda_{j+1}(t)}{\lambda_j(t)}\right)^2-\iota_j\iota_{j-1}\left( \frac{\lambda_j(t)}{\lambda_{j-1}(t)} \right)^2\right)\right|\leq C\tgamma^{3}(t),\\
  \label{Re30abs}
  |\ell|\leq C\left( \frac{\lambda_{\tilde{J}}(t)}{\lambda_{\tJ}(\tilde{t}_n)} \right)\sqrt{\gamma(t)}.
 \end{gather}
\end{proposition}

A contradiction is then obtained between the results of Proposition \ref{P:contraabs} and of Proposition \ref{A:SDE}, concluding the proof of Theorem \ref{T:main}.

\subsection{Proof of the rigidity result for global non-radiative solutions}

The proof of Theorem \ref{T:rigidity} is a direct consequence of the proof of the soliton resolution Theorem \ref{T:main}. We refer to \cite{DuKeMe19Pb}, Section 7, for the details.

\appendix

 \section{Study of a system of differential inequalities}
\label{A:SDE}
In this appendix we prove
\begin{proposition}
\label{P:SD1}
 Let $C>0$, $\kappa_0,\kappa_2>0$ and $J_0\geq 2$ an integer. There exists $\eps_0>0$ with the following property. For all $L>0$, there exists $T^*=T^*(L,C,\kappa_0,\kappa_2,J_0)$ such that, if $\beta_1,\ldots,\beta_{J_0}\in C^1([0,T],\RR)$, $\lambda_1,\ldots,\lambda_{J_0}\in C^1([0,T], (0,\infty))$, $(\iota_1,\ldots,\iota_{J_0})\in \{\pm 1\}^{J_0}$ and $\gamma(t)=\sup_{1\leq j\leq J_0-1} \frac{\lambda_{j+1}}{\lambda_j}$ satisfy, for $t\in [0,T]$
 \begin{gather}
  \label{SD10} \forall j\in \llbracket 1,J_0\rrbracket,\quad  \left|\lambda_j'-\kappa_2\beta_j\right|\leq C\gamma^{2},\\
  \label{SD11} \forall j\in \llbracket 1,J_0\rrbracket,\quad \left|\lambda_j \beta_j'+\kappa_0 \left( \iota_j\iota_{j+1} \left( \frac{\lambda_{j+1}}{\lambda_j} \right)^2- \iota_j\iota_{j-1}\left( \frac{\lambda_{j}}{\lambda_{j-1}} \right)^2\right)\right|\leq C\gamma^{3},\\
  \label{SD12}
  \frac{1}{C}\sup_{1\leq j\leq J_0} \beta_j^2(t)\leq \gamma^2(t)\leq \eps_0^2,\\
  \label{SD13}
  \left( \frac{\lambda_1(t)}{\lambda_1(0)} \right)^4\geq \frac{L}{\gamma^2(t)},
 \end{gather}
then
$$T\leq T^*\lambda_1(0).$$
\end{proposition}
In \eqref{SD11} we have made the convention $\lambda_0=+\infty$, $\lambda_{J+1}=0$.

The proof is the same as the proof of \cite[Proposition 6.1]{DuKeMe19Pb} (see Subsection 6.4 there). We sketch it for the sake of completeness. In \cite[Proposition 6.1]{DuKeMe19Pb}, \eqref{SD12} is replaced by a slightly stronger assumption. The only point to check is that the proof still works under the weaker assumption \eqref{SD12}.

\begin{remark}
 In the case where the nonlinearity is $u^2$, we will use Proposition \ref{P:SD1} with $\iota_j=+1$ for all $j$. The general case $\iota_j\in \{\pm 1\}$ is needed for the odd nonlinearity $|u|u$.
\end{remark}
\begin{proof}
In all the proof $C$ denotes a large positive constant that might depend on $\kappa_0$ and $\kappa_2$ and the constant $C$ in \eqref{SD10}, \eqref{SD11}, \eqref{SD12}, but not on $L$ and $\eps_0$. 

Rescaling the time and normalizing the parameters $\lambda_j$, we can assume $\lambda_1(0)=1$. We define $(\theta_j)_{1\leq j\leq J_0}$ by $\theta_1=1$ and
 $\theta_j=2\theta_{j-1}$if $\iota_j\iota_{j-1}=1$, $\theta_j=\frac{1}{2}\theta_{j-1}$ if $\iota_j\iota_{j-1}=-1$, so that
 \begin{equation}
\label{thetajj}
 \iota_j\iota_{j-1}(\theta_j-\theta_{j-1})=c_j\theta_{j-1},\quad c_j\in\left\{\frac{1}{2},1\right\}.
 \end{equation}

 Let
$$ A(t)=\sum_{j=1}^{J_0} \theta_j\lambda_j(t)\beta_j(t).$$

\noindent\textbf{Step 1.} In this step we prove
\begin{align}
 \label{SD21}
 A'(t)&\geq \kappa_2\sum_{j=1}^{J_0}\theta_j\beta_j^2(t)+\frac{1}{C}\gamma^2(t),\\
\label{SD22}
 A'(t)&\geq \frac{1}{\kappa_2}\sum_{j=1}^{J_0}\theta_j\left(\lambda_j'(t)\right)^2+\frac{1}{C}\gamma^2(t).
 \end{align}
By \eqref{SD10}, \eqref{SD11} and \eqref{SD12}, reorganizing the indices in the second sum, we obtain
$$ A'(t)=\kappa_2\sum_{j=1}^{J_0} \theta_j\beta_j^2+\kappa_0\sum_{j=2}^{J_0} \left( \frac{\lambda_j}{\lambda_{j-1}} \right)^2\iota_j\iota_{j-1}(\theta_j-\theta_{j-1})+\OOO(\gamma^{3}),$$
and \eqref{SD21} follows from \eqref{thetajj} and \eqref{SD12}. In view of \eqref{SD10}, we also obtain \eqref{SD22}.\\

\noindent\textbf{Step 2.} Let
\be \label{SD1000}
V(t)=\sum_{j=1}^{J_0} \theta_j\lambda_j^2(t).
\ee
In this step we prove
\begin{equation}
 \label{SD30}
 A\left( \frac 1L \right)\geq \frac 1C L^4,\quad V\left( \frac{1}{L} \right)\leq \frac{5}{L^2}.
\end{equation}
Indeed, using that $A'(t)\geq C^{-1}\gamma^2(t)$ by the preceding step, we obtain
\begin{equation}
 \label{SD32}
 \int_0^{1/L} A'(t)\lambda_1^4(t)dt\geq \frac 1C \int_0^{1/L} \gamma^2(t)\lambda_1^4(t)dt.
\end{equation}
Next, by integration by parts and using that $\lambda_1(0)=1$,
\begin{equation}
 \label{SD33}
 \int_{0}^{1/L} A'(t)\lambda_1^4(t)=A\left( \frac{1}{L} \right)\lambda_1^4\left( \frac{1}{L} \right)-A(0)-4\int_0^{1/L} A(t)\lambda_1^3(t)\lambda_1'(t)dt.
\end{equation}
Since by \eqref{SD10} and \eqref{SD12}, $A=\frac{1}{\kappa_2} \lambda_1'\lambda_1+\OOO(\gamma^{2}\lambda_1)$, we obtain
\begin{equation}
 \label{SD34}
 \int_0^{1/L} A(t)\lambda_1^3(t)\lambda_1'(t) dt=\frac{1}{\kappa_2} \int_0^{1/L} \lambda_1^4(t)(\lambda_1'(t))^2dt+\OOO\left( \int_0^{1/L} \gamma^{3}\lambda_1^4(t)dt \right).
\end{equation}
Combining \eqref{SD32}, \eqref{SD33} and \eqref{SD34} we deduce, using that $\gamma\leq \eps_0$ is small and \eqref{SD13},
\begin{multline}
 \label{SD35}
 A\left( \frac{1}{L} \right)\lambda_1^4\left( \frac 1L \right)\\
 \geq A(0) +\frac{4}{\kappa_2} \int_0^{1/L} \lambda_1^4(t)(\lambda_1'(t))^2dt+\frac{1}{C} \int_0^{1/L} \gamma^2(t)\lambda_1^4(t)dt\geq A(0)+\frac 1C.
\end{multline}
On the other hand, $|A(0)|=\left|\sum_{j=1}^{J_0} \theta_j\beta_j(0)\lambda_j(0)\right|\leq C\gamma(0)$ since $\lambda_1(0)=1$. Furthermore $\lambda_1(1/L)=1+\int_0^{1/L} \lambda_1'(t)dt=1+\OOO(L^{-1} \gamma(L^{-1}))$. Since by \eqref{SD13} at $t=0$, $L\leq \gamma^2(0)\leq \eps_0$, we have $L^{-1}\geq 1$, thus $\lambda_1(1/L)\leq 2/L$. This yields the first inequality in \eqref{SD30} in view of \eqref{SD35}. This also yields the second inequality in \eqref{SD30} since:
$$
V(\frac 1L) =\lambda_1^2(\frac 1L) \left[1+\sum_{j=2}^{J_0} \theta_j \frac{\lambda_j^2(\frac 1L)}{\lambda_1^2(\frac tL)} \right]=\lambda_1^2(\frac 1L) \left[1+O(\gamma^2(\frac tL)) \right]\leq \frac{5}{L^2}.
$$

\noindent\textbf{Step 3.} In this step, we prove that there exists $c_0> 1/2$ such that
\begin{equation}
 \label{SD40}
 \forall c\in [1/2,c_0],\; \forall t\in [0,T],\quad \frac{d}{dt}\left( \frac{A(t)}{V^c(t)} \right)\geq 0.
\end{equation}
To prove it, first notice using \eqref{SD21}, \eqref{P:SD1} and \eqref{SD12} that:
\begin{equation}
 \label{SD999}
A'(t)\geq \max \left( (\kappa_2+\frac 1C) \sum_{j=1}^{J_0} \theta_j \beta_j^2(t) \ , \ (\frac{1}{\kappa_2}+\frac 1C) \sum_{j=1}^{J_0} \theta_j (\lambda_j'(t))^2 \right).
\end{equation}
Then, since $\frac{d}{dt} \left( \frac{A(t)}{V^c(t)} \right)=\frac{A'V-cAV'}{V^{c+1}}$, the inequality \eqref{SD40} follows from
$$A(t)V'(t)=2\sum_{j=1}^{J_0} \theta_j\lambda_j\beta_j\sum_{j=1}^{J_0} \theta_j\lambda_j\lambda_j'\leq \frac{2 V(t)A'(t)}{\sqrt{ (\kappa_2+C^{-1})(\kappa_2^{-1}+C^{-1})}},$$
where we used Cauchy-Schwarz followed by \eqref{SD1000} and \eqref{SD999}.\\

\noindent \textbf{Step 4.} \emph{Conclusion}. From \eqref{SD40} with $c=c_0$ and Step 2, for all $t\in [1/L,T]$, $\frac{A(t)}{V^{c_0}(t)}\geq \frac{1}{C}L^{4+2c_0}$. Hence $V(t)^{c_0}\leq CL^{-4-2c_0}A(t)\leq CL^{-4-2c_0} V(t)^{1/2}\gamma(t)$, by Cauchy-Schwarz and \eqref{SD12}, which yields
\begin{equation}
 \label{SD41} 
 V(T)\leq L^{-\frac{4+2c_0}{c_0-1/2}}.
\end{equation} 
Next, from \eqref{SD40} with $c=1/2$ and Step 2, 
for $t\in [1/L,T]$, 
$\frac{A(t)}{V^{1/2}(t)}\geq \frac{1}{C}L^{5}$. Since by \eqref{SD10} and \eqref{SD12}, $V'(t)=2\kappa_2\sum_{j=1}^{J_0}\theta_j\beta_j\lambda_j+\OOO(\gamma^{2}V^{1/2}(t))$, and $A(t)=\sum_{j=1}^{J_0}\theta_j\beta_j\lambda_j$, we deduce $\frac{V'(t)}{V(t)^{1/2}}\geq \frac 1CL^5+\OOO(\gamma^{2})$, and thus, integrating between $1/L$ and $t$,
$$ \sqrt{V(t)}\geq \frac{1}{C}L^5\left( t-\frac{1}{L} \right)+\int_{1/L}^{t}\OOO(\gamma^{2}(s))ds.$$
By Steps 1 and 2, $\int_{1/L}^{t}\gamma^2(s)ds\leq CA(t)\leq C \sqrt{V(t)}\eps_0$. Combining these two inequalities:
\begin{equation}
\label{SD42} 
\sqrt{V(T)}\geq \frac{1}{C}L^5\left( T-\frac{1}{L} \right).
\end{equation} 
Combining \eqref{SD41} and \eqref{SD42}, we obtain as desired an upper bound for $T$ that only depends on $L$. 
\end{proof}

\section{A few estimates}
\label{A:estimates}
\begin{lemma}
Let $0<\lambda<\mu$ and $R>0$. Then
\begin{gather}
\label{estim1}
\|\Lambda W_{[\lambda]}\|_{L^2_R}\approx \|W_{[\lambda]}\|_{L^2_R}\approx \min(1,\lambda/R), \quad \left|\int_{|x|>R}(\Lambda W)_{[\lambda]}(\Lambda W)_{[\mu]}dx\right|\lesssim \frac{\lambda}{\mu} \\
\label{estim2}
\Big\|\big|\Lambda W_{(\lambda)}W_{(\mu)}\big|+\left|\Lambda W_{(\mu)}W_{(\lambda)}\right|+W_{(\lambda)}W_{(\mu)}\Big\|_{L^1L^2(\{|x|>|t|\})}\lesssim \frac{\lambda^2}{\mu^2}\left\langle\log(\frac{\mu}{\lambda})\right\rangle.\\
 \label{estim3}
  \big\|t(\Lambda W)_{[\lambda]}W_{(\mu)}\big\|_{L^1L^2(\{|x|>|t|\})}\lesssim \frac{\lambda}{\mu},\quad \big\|t(\Lambda W)_{[\mu]}W_{(\lambda)}\big\|_{L^1L^2(\{|x|>|t|\})}\lesssim \frac{\lambda^2}{\mu^2}
\end{gather}
If $\lambda<\mu<R$,
\begin{equation}
 \label{estim4}
 \left\|t(\Lambda W)_{[\mu]}W_{(\lambda)}\right\|_{L^1L^2(\{|x|>R+|t|\})}\lesssim \frac{\lambda^2\mu}{R^3}.
\end{equation} 
If $R<R'<\lambda$, 
\begin{equation}
 \label{estim5}
 \Big\|W_{(\lambda)} \indic_{\{R+|t|<|x|<R'+|t|\}}\Big\|_{L^2L^4}\lesssim \left(  \frac{R'-R}{\lambda}\right)^{1/4}.
\end{equation}
If $R\geq 1$:
\begin{equation} \label{estim6}
\| W  \indic_{\{\max(|t|,R)<|x|\}} \|_{L^2L^4}\lesssim R^{-2}.
\end{equation}

\end{lemma}
\begin{proof}
 The proof is by direct computations, using that $W$ and $\Lambda W$ are bounded and of order $1/|x|^4$ at infinity. The estimates \eqref{estim1} follow immediately. 
 
 Note that we can always assume $\mu=1$ by scaling. To prove \eqref{estim4} and the second inequality in \eqref{estim3}, observe that
 $$ \left\|t\Lambda W\,W_{(\lambda)}\right\|_{L^1L^2(\{|x|>R+|t|\})}\lesssim \frac{\lambda^2}{\mu^2} \left\|\frac{t\Lambda W}{r^4}\right\|_{L^1L^2(\{|x|>|t|+R\})},$$
and the inequalities follow, 
using that by direct computations
$$\frac{t\Lambda W}{r^4}\indic_{\{|x|>|t|\}}\in L^1L^2,\quad \left\|\frac{t\Lambda W}{r^4}\indic_{\{|x|>|t|+R\}}\right\|_{L^1L^2}\lesssim 1/R^3$$ 
for large $R$. The proof of the first inequality in \eqref{estim3} in the same. To prove \eqref{estim2}, we write
\begin{multline*}
 \int_{0}^{\infty} \left( \int_{\{|x|>t\}}W_{(\lambda)}^2W^2dx \right)^{1/2}dt\lesssim \int_{0}^{\lambda} \left(\int_{t}^{\infty}\frac{1}{\lambda^4}W^2\left(\frac{r}{\lambda}\right)r^5dr\right)^{1/2}dt\\
 +
\int_{\lambda}^{1} \left(\int_{t}^{\infty}\frac{\lambda^4}{r^8}r^5dr\right)^{1/2}dt+\int_{1}^{\infty} \left( \int_{t}^{\infty} \frac{\lambda^4}{r^{16}}r^5dr \right)^{1/2}dt\lesssim \lambda^2+\lambda^2|\log\lambda|+\lambda^2.
\end{multline*}
The proof of the estimates of $\Lambda W_{(\lambda)}W_{(\mu)}$ and $\Lambda W_{(\mu)}W_{(\lambda)}$ are the same.

We sketch the proof of \eqref{estim5}. By scaling, we can assume $\lambda=1$. Then
\begin{multline*}
\left\|W \indic_{\{R+|t|<|x|<R'+|t|\}}\right\|_{L^2L^4}\lesssim \left\|\indic_{\{|t|<1\}} \indic_{\{R+|t|<|x|<R'+|t|\}}\right\|_{L^2L^4}\\
+\left\|\frac{1}{|x|^4}\indic_{\{|t|>1\}} \indic_{\{R+|t|<|x|<R'+|t|\}}\right\|_{L^2L^4}\lesssim \left( \frac{R'-R}{\lambda} \right)^{1/4}.
\end{multline*}
To prove \eqref{estim6}, we decompose:
\begin{multline*}
\| W  \indic_{\{\max(|t|,R)<|x|\}} \|_{L^2L^4}^2 =\int_{|t|\leq R}\| W \indic_{\{|x|\geq R\}} \|_{L^4}^2 dt+\int_{|t|\geq R}\| W \indic_{\{|x|\geq |t|\}} \|_{L^4}^2 dt\\
\lesssim \int_{|t|\leq R}R^{-5}dt+\int_{|t|\geq R}|t|^{-5} dt \lesssim R^{-4}.
\end{multline*}
\end{proof}

\bibliographystyle{alpha} 
\bibliography{toto}
\end{document}